\documentclass[12pt]{amsart}

\usepackage{amsmath}
\usepackage{amssymb}

\theoremstyle{plain}
\newtheorem{lemma}{Lemma}[section]
\newtheorem{proposition}[lemma]{Proposition}
\newtheorem{theorem}[lemma]{Theorem}
\newtheorem{corollary}[lemma]{Corollary}

\theoremstyle{definition}
\newtheorem{definition}[lemma]{Definition}
\newtheorem{remark}[lemma]{Remark}
\newtheorem{example}[lemma]{Example}

\usepackage{amscd,stmaryrd,MnSymbol,enumerate}
\usepackage{hyperref}

\newcommand{\rk}{\operatorname{rk}}

\newcommand{\ZZ}{\mathbb{Z}}
\newcommand{\CC}{\mathbb{C}}
\newcommand{\MM}{\mathcal{M}}
\newcommand{\Ad}{\operatorname{Ad}}

\newcommand{\tr}{\operatorname{tr}}
\newcommand{\SL}{\mathrm{SL}}
\newcommand{\sln}[1][n]{\mathfrak{sl}(#1)}
\newcommand{\GLn}[1][n]{\mathrm{GL}(#1)}

\newcommand{\metrep}{\varrho}

\def\co{\thinspace\colon\thinspace}
\title{$\mathrm{SL}_n(\mathbb{C})$--representation spaces of knot groups }

\author[M. Heusener]{Michael Heusener}
\address{
Universit\'e Clermont Auvergne, Universit\'e Blaise Pascal, Laboratoire de Math\`ematiques, 
BP 10448, F-63000 Clermont-Ferrand \\ CNRS, UMR 6620, LM, F-63178 Aubiere, France}
\email{Michael.Heusener@math.univ-bpclermont.fr}
\thanks{Received 
January 31, 2016}

\subjclass[2010]{57M25; 57M05; 57M27; 20C99, 14M99} 
\keywords{knot group; representation variety, character variety}

\begin{document}

\maketitle
\begin{abstract}
%
The first part of this article is a general introduction to the the theory of 
representation spaces of discrete groups into $\SL_n(\CC)$.
Special attention is paid to knot groups.
In Section~\ref{sec:background} we discuss the difference between the tangent space at the representation variety, and the representation scheme. We give an example of Lubotzky and Magid  of a non scheme reduced representation (see Example~\ref{ex:LM}).

In the second part recent results about the 
representation and character varieties of knot groups into 
$\mathrm{SL}_n(\mathbb{C})$ with $n\geq 3$ are presented. 
This second part concerns mostly joint work 
with L.\ Ben Abdelghani, O.\ Medjerab, V.\ Mu\~nos and J.\ Porti.
\end{abstract}

\section{Introduction }
\label{sec:introduction}

Since the foundational work of Thurston  \cite{ThurstonNotes, Thu82}  and Culler and Shalen \cite{CS83}, 
the varieties of representations and characters of three-manifold groups in $\SL_2(\CC)$
have been intensively studied, as they reflect geometric and topological properties of the three-manifold.
In particular they have been  used to study knots $k\subset S^3$, by analysing the
$\SL_2(\CC)$-character variety of the fundamental group of the knot complement
$S^3-k$ (these are called \emph{knot groups}). 

Much less is known about the character varieties of three-manifold groups in other Lie groups, notably
for $\SL_n(\CC)$ with $n\geq 3$. There has been an increasing interest for those
in the last years. For instance, inspired by the A-coordinates in higher 
Teichm\"uller theory of Fock and Goncharov \cite{FG06}, some authors have used
the so called Ptolemy coordinates for studying spaces of representations,
based on subdivisions of ideal triangulations of the three-manifold. Among others, 
we mention the work of  Dimofty, Gabella, Garoufalidis, Goerner, Goncharov, Thurston, 
and Zickert \cite{DGG, DG, GTZ, GZ, GGZ}.
Geometric aspects of these representations, including volume and rigidity, have been addressed by 
Bucher, Burger, and Iozzi in \cite{BBI},
and by Bergeron, Falbel, and Guilloux in \cite{BFG}, who view these representations as holonomies of marked  flag
structures. We also recall the work 
Deraux and Deraux-Falbel in \cite{D1, D2, DF} to study CR and complex hyperbolic structures.
 
\subsection*{Acknowledgements} I like to thank the organizers of the workshop
\emph{Topology, Geometry and Algebra 
of low-dimensional manifolds} for inviting me. In particular, I am very grateful to  Teruaki Kitano and
Takayuki Morifuji for their hospitality (before and after the workshop).
Also, thanks to Leila Ben Abdelghani, Paul Kirk, and Simon Riche for very
helpful discussions during the writing of this article.
Also, I am pleased to acknowledge support from the ANR projects SGT and ModGroup.


\section{Background}\label{sec:background}

\begin{definition}
Let $k\subset S^3$ be a knot. The \emph{knot group} is $\Gamma_k := \pi_1(S^3 \setminus k)$.
The knot exterior is the compact manifold $C_k = \overline{S^3\setminus V(k)}$ where $V(k)$ is a tubular neighborhood of $k$. 
\end{definition}
In what follows we will make use of the following properties of knot groups:
\begin{itemize}
\item
We have 
$H_1(C_k;\ZZ)\cong\mathbb Z$. A canonical surjection
$\varphi\co \Gamma_k\to\ZZ$ is given by $\varphi(\gamma)=\mathrm{lk}(\gamma,k)$
where $\mathrm{lk}$ denotes the linking number in $S^3$ (see \cite[3.B]{BZH}).

\item
The knot exterior is aspherical: we have $\pi_n(C_k)=0$ for $n>1$ i.e. $C_k$ is an Eilenberg--MacLane space $K(\Gamma_k,1)$ (see \cite[3.F]{BZH}).
 As a cosequence, the  \mbox{(co-)}homology groups of 
$\Gamma$ and $C_k$ are naturally identified, and for a given $\Gamma_k$-module $M$ we have
$H^*(C_k;M)\cong H^*(\Gamma_k;M)$, and $H_*(C_k;M)\cong H_*(\Gamma_k;M)$.
\end{itemize}

It follows that every \emph{abelian} representation factors through 
$\varphi\co \Gamma_k\to\ZZ$. Here we call $\rho$ \emph{abelian} if its image is abelian.
We obtain for each non-zero complex number $\eta\in\CC^*$ an abelian representation
$\eta^\varphi\co\Gamma_k\to\mathrm{GL}(1,\mathbb{C}) =\CC^*$ given by
$\gamma\mapsto \eta^{\varphi(\gamma)}$.

\subsection{Representation varieties}
The general reference for representation and character varieties is 
 Lubotzky's and Magid's book \cite{LM85}. Let $\Gamma=\langle\gamma_1,\ldots,\gamma_m\rangle$ be a finitely generated group. 
 
\begin{definition}
A $\mathrm{SL}_n(\mathbb{C})$-representation is a homomorphism $\rho\co\Gamma\to\mathrm{SL}_n(\mathbb{C})$.
The 
$\mathrm{SL}_n(\mathbb{C})$-representation variety is
\[
R_n(\Gamma) = \mathrm{Hom} (\Gamma,\mathrm{SL}_n(\mathbb{C}))
\subset 
\mathrm{SL}_n(\mathbb{C})^m\subset M_n(\CC)^m\cong\CC^{n^2m}\,.
\]
\end{definition}

The representation variety 
$R_n(\Gamma)$ is an affine algebraic set. It is contained in
$\SL_n(\CC)^m $ via the inclusion 
$\rho\mapsto\big(\rho(\gamma_1),\ldots,\rho(\gamma_m)\big)$, and
it is
the set of solutions of a system of polynomial equations in the matrix coefficients.

\subsubsection{Affine algebraic sets} \label{subsubsec:affinealg} Let $k$ be a field and let $F_\lambda=F_\lambda(x_1,\ldots,x_n)\in k[x_1,\ldots,x_n]$, $\lambda\in \Lambda$, be a family of polynomials.
 The set of all common zeros of this family of polynomials is denoted by
 \[
 \mathcal V(\{F_\lambda, \lambda\in\Lambda\}) =\{ v\in k^n\mid F_\lambda(v)=0 \text{ for all $\lambda\in\Lambda$}\}\,.
 \]
 It is clear that $ \mathcal V(\{F_\lambda, \lambda\in\Lambda\})=  \mathcal V(I)$ where $I=(\{F_\lambda, \lambda\in\Lambda\})$ is the ideal generated by the family 
 $\{F_\lambda\}_{\lambda\in\Lambda}$. Recall that, by Hilbert's basis theorem, each ideal $I\subset k[x_1,\ldots,x_n]$ has a finite set of generators. 
An \emph{(affine) algebraic subset} in $k^n$ is a subset $V\subset k^n$  consisting of all common zeros of finitely many polynomials with coefficients in $k$.
It is easy to see that arbitrary intersections and finite unions of affine algebraic sets are affine algebraic.

Now, given an algebraic subset $V\subset k^n$ a function $f\co V\to k$  is called \emph{regular} if there exists $F\in k[x_1,\ldots,x_n]$ such that $f(v)=F(v)$ for all $v\in V$.
All regular functions on $V$ form the \emph{coordinate ring} $\mathcal{O}(V)$ (or $k[V]$) of the variety $V$. Notice that $\mathcal{O}(V)$ is a finitely generated $k$ algebra since
there is a surjection $k[x_1,\ldots,x_n]\to \mathcal{O}(V)$. The kernel of this surjection is called the \emph{ideal of $V$} and is denoted by $\mathcal{I}(V)$, hence
\[ 
\mathcal{I}(V) =\{ F\in k[x_1,\ldots,x_n]\mid F(v)=0 \text{ for all $v\in V$} \}
\quad\text{ and }\quad \mathcal{O}(V)\cong k[x_1,\ldots,x_n]/\mathcal{I}(V)\,.
\]

Notice that in general $\mathcal{I}\big(\mathcal{V}(I)\big)\supset I$ but 
$\mathcal{I}\big(\mathcal{V}(I)\big)\neq I$ is possible. For example, if $V\subset k$ is given by the equation $x^2=0$ then
$I=(x^2)\subsetneq (x) = \mathcal{I}(\{0\})$. 
If $k$ is algebraically closed then Hilbert's Nullstellensatz implies  that
$\mathcal{I}\big(\mathcal{V}(I)\big)$ is equal to the \emph{radical} $\sqrt{I}$ of $I$
\[
\mathcal{I}\big(\mathcal{V}(I)\big)=\sqrt{I} = \{F\in k[x_1,\ldots,x_n]\mid \text{$\exists m\in\ZZ$, $m>0$, such that $F^m\in I$}\}\,.
\]
Now, two affine algebraic sets $V\subset k^m$ and $W\subset k^n$ are isomorphic if and only if there is an algebra isomorphism between
$\mathcal{O}(V)$ and $\mathcal{O}(W)$ (see \cite{Shaf1} for more details).

\begin{example} If $V=\{v=(v_1,\ldots,v_n)\}\subset k^n$ is a point then $\mathcal{I}(\{v\}) = (x_1-v_1,\ldots, x_n-v_n)$ and 
$\mathcal{O}(V)\cong k$. In general, $\mathcal{O}(V)$ is finite dimensional as a $k$-vector space if and only if $V$ is finite, and in this case
$\dim_k(\mathcal{O}(V))= \#  V$.
\end{example}

In the next example we investigate some very special representation varieties:
\begin{example} A homomorphism of $\ZZ$ is determined by the image of the generator $1\in\ZZ$ and hence 
$R_n(\mathbb{Z}) \cong \mathrm{SL}_n(\CC)$. 
Similar, for a free group $F_k$ of rank $k$ we have
$R_n(F_k)\cong \mathrm{SL}_n(\CC)^k$. 

For the cyclic group $\ZZ/2\ZZ$ of two elements  we have $R_2(\ZZ/2\ZZ)=\{\pm I_2\}$ consists only of two points the identity matrix $I_2$ and 
$-I_2$. Hence, $R_2(\ZZ/2\ZZ)$ is not irreducible as an algebraic variety. Even more concretely, a representation $\rho\co\ZZ/2\ZZ\to\SL_2(\CC)$ is determined by the image $X$ of a generator. Now, considering $X=\big(\begin{smallmatrix} x_{11} & x_{12} \\ x_{21} & x_{22} \end{smallmatrix}\big)\in M_2(\CC[x_{11},x_{12},x_{21},x_{22}])$ the relation $X^2 = \big(\begin{smallmatrix} 1 & 0 \\ 0 & 1 \end{smallmatrix}\big)$ and $\det X =1$ give the equations $x_{11} x_{22} - x_{12}x_{21} =1$, and
\[ 
x_{11}^{2} + x_{12} x_{21} =1,\quad 
x_{11} x_{12} + x_{12} x_{22} =0,\quad
x_{11} x_{21} + x_{21} x_{22}=0,\quad x_{12} x_{21} + x_{22}^{2} =1\,.
\] 
The ideal $I$ has a much simple        r set of generators: 
$I=(x_{22}^2 - 1, x_{11} - x_{22}, x_{12}, x_{21})$, and hence
\[
\CC[x_{11},x_{12},x_{21},x_{22}]/I\cong \CC[x]/(x^2 -1)\cong \CC[x]/(x -1)\oplus \CC[x]/(x+1)\cong\CC\oplus\CC
\]
is the coordinate ring of the union of two points. 
\end{example}

\subsubsection{General facts}
Given two representations $\rho_1\co\Gamma\to\mathrm{GL}_m(\CC)$ and
$\rho_2\co\Gamma\to\mathrm{GL}_n(\CC)$ we define the \emph{direct sum}
$\rho_1\oplus\rho_2\co\Gamma\to\mathrm{GL}_{m+n}(\CC)$ 
and the \emph{tensor product} 
$\rho_1\otimes\rho_2\co\Gamma\to\mathrm{GL}_{mn}(\CC)$ by
\[
\big(\rho_1\oplus\rho_2\big)(\gamma) = 
\left(\begin{array}{c|c} \rho_1(\gamma) & 0 \\ \hline 0 &\rho_2(\gamma)\end{array}\right)
\quad\text{ and }\quad
\big(\rho_1\otimes\rho_2\big)(\gamma) = \rho_1(\gamma)\otimes\rho_2(\gamma),
\quad \forall\, \gamma\in \Gamma,
\]
 respectively. Here, $A\otimes B$ denotes the \emph{Kronecker product} of $A\in\mathrm{GL}_m(\CC)$ and 
 $B\in\mathrm{GL}_n(\CC)$.
 The \emph{dual representation} $\rho^*\co\Gamma\to\mathrm{GL}(n)$ of 
 $\rho\co\Gamma\to\mathrm{GL}(n)$ is defined by 
 $\rho^*(\gamma)   = \mbox{}^t\rho(\gamma)^{-1}$ where $\mbox{}^tA$ is the transpose of the matrix $A$. (See also Lemme~\ref{lem:dual}.)

\begin{definition}\label{def:irreducibleRep}
We call a representation $\rho\co\Gamma\to\mathrm{GL}_n(\mathbb{C})$ \emph{reducible} if there exists a nontrivial subspace $V\subset\CC^n$, $0\neq V\neq \CC^n$, such that $V$ is 
$\rho(\Gamma)$-stable. The representation 
$\rho$ is called \emph{irreducible} if it is not reducible.
A \emph{semisimple} representation is a direct sum of irreducible representations.
\end{definition}

The group $\mathrm{SL}_n(\CC)$ acts by conjugation on $R_n(\Gamma)$.
More precisely, for $A\in\mathrm{SL}_n(\mathbb{C})$ and $\rho\in R_n(\Gamma)$ we define
$(A.\rho) (\gamma) = A \rho(\gamma) A^{-1}$ for all $\gamma\in\Gamma$.
Moreover, we let
$\mathrm O(\rho)=\{A.\rho\mid A\in\mathrm{SL}_n(\mathbb{C})\}$ 
denote the \emph{orbit} of $\rho$.
In what follows we will write $\rho\sim\rho'$ if there exists 
$A\in\mathrm{SL}_n(\CC)$ such that $\rho'=A.\rho$, and we will call $\rho$ and $\rho'$ \emph{equivalent}.
For $\rho\in R_n(\Gamma)$ we define its \emph{character}
$\chi_\rho\co\Gamma\to\CC$ by $\chi_\rho(\gamma)= \mathrm{tr}(\rho(\gamma))$.
We have $\rho\sim\rho' \Rightarrow\chi_\rho = \chi_{\rho'}$.

\begin{lemma}
Let $\rho\in R_n(\Gamma)$ be a representation.
The orbit $\mathrm O(\rho)$ is closed if and only if $\rho$ is semisimple.
Moreover, let $\rho,\rho'$ be semisimple. 
Then $\rho\sim\rho' $ if and only if $\chi_\rho =\chi_{\rho'}$. 
\end{lemma}
\begin{proof}
See Theorems~1.27 and  1.28 in Lubotzky's and Magid's book \cite{LM85}.
\end{proof}

\newpage
\begin{example} \label{ex:non-semisimple}We give two examples of a non-semisimple representations:
\begin{enumerate}
\item Let $\rho\co\ZZ\to\mathrm{SL}_2(\CC)$ be given by
$\rho(n) = \big(\begin{smallmatrix}1&n\\0 & 1\end{smallmatrix}\big)$.
The representation $\rho$ is reducible but not semisimple.
Notice that the orbit $\mathrm{O}(\rho)$ is not closed, 
$\overline{\mathrm{O}(\rho)}$ contains the trivial representation:
$
\lim_{t\to0}  \big(\begin{smallmatrix} t &0\\0 & t^{-1}\end{smallmatrix}\big)
 \big(\begin{smallmatrix} 1 &n\\0 & 1\end{smallmatrix}\big)
 \big(\begin{smallmatrix} t^{-1} &0\\0 & t\end{smallmatrix}\big)=
 \big(\begin{smallmatrix} 1 &0\\0 & 1\end{smallmatrix}\big)
$.

\item Let $\Gamma = \langle S,T\mid STS = TST\rangle$ be the group of the trefoil knot, and let 
$\zeta\in\CC$ be a primitive $12$-th root of unity, $\zeta^4 - \zeta^2 + 1 = 0$. 
We define 
$\rho\co\Gamma\to\mathrm{SL}_2(\CC)$  by
$\rho(S) = \big(\begin{smallmatrix}\zeta & 0 \\ 0 & 1/\zeta\end{smallmatrix}\big)$, and 
$\rho(T) = \big(\begin{smallmatrix} \zeta & 1 \\ 0 & 1/\zeta\end{smallmatrix}\big)$. 
The representation is reducible but not semisimple. Again the orbit 
$\mathrm{O}(\rho)$ is not closed, 
$\overline{\mathrm{O}(\rho)}$ contains the diagonal representation 
$\rho_d = \zeta^\varphi\oplus\zeta^{-\varphi}$.
\end{enumerate}
\end{example}

\subsection{Character varieties} 

The \emph{algebraic quotient} or \emph{GIT quotient} for the action of $\mathrm{SL}_n(\CC)$ on
$R_n(\Gamma)$ is called the \emph{character variety}. This quotient  will be denoted by
$X_n(\Gamma) = R_n(\Gamma)\sslash\mathrm{SL}_n(\CC)$.
The character variety is not necessary an irreducible affine algebraic set.

For an introduction to algebraic invariant theory see Dolgachev's book \cite{Dol04}.
Roughly speaking, \emph{geometric invariant theory} is concerned with an algebraic  action of a 
group $G$ on an algebraic variety $V$. Classical invariant theory addresses the situation when 
$V$ is a vector space and $G$ is either a finite group, or one of the classical Lie groups that acts linearly on $V$. The action of $G$ on $V$ induces an action of $G$ on the coordinate algebra
$\mathcal{O}(V)$ of the variety $V$ given by
$ g\cdot f(v)=f(g^{-1}v)$, for all $g\in G$, and $v\in V$.

The invariant functions of the $G$-action on $V$ are 
\[
\mathcal{O}(V)^G =\{ f\in \mathcal{O}(V)\mid g\cdot f = f \text{ for all $g\in G$} \}\,.
\]
The invariant functions $\mathcal{O}(V)^G$ form a commutative algebra, and this algebra is interpreted as the algebra of functions on the \emph{GIT quotient} $V\sslash G$. 
The main problem is to prove that the algebra $\mathcal{O}(V)^G$ is finitely generated. This is necessary if one wantes the quotient to be an affine algebraic variety. We are only interested in 
\emph{affine varieties} $V$ and in \emph{reductive groups} $G$, and in this situation Nagata's theorem applies (see \cite[Sec.~3.4]{Dol04}). Reductive groups include all finite groups and all classical groups (see \cite[Chap.~3]{Dol04}).
Geometrically,  the GIT quotient $V\sslash G$ parametrizes the set of closed orbits
(see \cite[Corollary~6.1]{Dol04}).
For a point $v\in V$ the orbit $G\,v$ will be denoted by $\mathrm O(v)$. If $f_1,\ldots,f_N$ generate the algebra
$\mathcal{O}(V)^G$ then a model for the quotient is given by the image of the map
$t\co V\to V\sslash G\subset\CC^N$ given by $t(v) = (f_1(v),\ldots,f_N(v))$.

\begin{example} \label{ex:GIT}
We will give three basic examples of  GIT quotients:
\begin{enumerate}
\item Let $\CC^*$ act on $\mathbb {C}^2$ by 
$\lambda.(z_1,z_2) = (\lambda z_1, \lambda z_2)$. The topological quotient
$\CC^2/\CC^*$ is a non-Hausdorff topological space. More precisely, only the  orbit  
$\mathrm O(0,0)=\{(0,0)\}$ is closed, and $(0,0)$ is contained in the closure of every orbit. 
The algebra $\mathcal{O}(\CC^2)$ is isomorphic to 
the polynomial ring in two variables $\CC[x_1,x_2]$, and $\CC[x_1,x_2]^{\CC^*}$ consist only of the constant functions i.e.\ 
$\CC[x_1,x_2]^{\CC^*}\cong\CC$. 
The GIT quotient $\CC^2\sslash\CC^*\cong\{\ast\}$ is just one point, and  $\dim( \CC^2\sslash \CC^*)< \dim(\CC^2) - \dim( \CC^*)$.

\item\label{ex:GIT-C^*}
Let $\CC^*$ act on $\CC^2$ by 
$\lambda.(z_1,z_2) = (\lambda\, z_1, 1/\lambda\,z_2)$. The topological quotient
$\CC^2/\CC^*$ is again non-Hausdorff topological space. More precisely,  
$\mathrm O(1,0)$ and $\mathrm O(0,1)$ are not closed and disjoint, but the closed orbit 
$\{(0,0)\}$ is contained in the closure of both orbits. 
Now, in order to determine $\CC[x_1,x_2]^{\CC^*}$ we consider 
the space $R_n\subset \CC[x_1,x_2]$ of homogeneous polynomials of degree $n$.
The set $R_n$ is a vector space of dimension $n+1$ with basis $x_1^i x_2^j$,
$i+j=n$, and it is stable by the action of $\CC^*$. Now, 
$\lambda\cdot (x_1^i x_2^{j}) = \lambda^{i-j}\, x_1^i x_2^{j}$, and  the algebra of invariant functions is generated by $x_1x_2$. Hence $\CC[x_1,x_2]^{\CC^*}=\CC[x_1x_2]\cong\CC[x]$.
It follows that GIT quotient $\CC^2\sslash\CC^*\cong\CC$. 
The quotient map $t\co\CC^2\to\CC$ is given by the invariant functions $t(z_1,z_2)=z_1z_2$.
The whole ``non-hausdorff'' part 
$
\mathrm O(1,0)\cup \{(0,0)\}\cup \mathrm O(0,1)\cong \CC\times\{0\}\cup\{0\}\times\CC
$ 
is identified, and mapped by $t$ onto $0\in\CC$.
\item
 $\mathrm{SL}_n(\CC)$ acts on itself by conjugation. Two matrices are conjugate if and only if they have the same Jordan normal forms. As we already saw in Example~\ref{ex:non-semisimple},
 the orbit of an unipotent element  is in general  not closed.
 The GIT quotient $\mathrm{SL}_n(\CC)\sslash\mathrm{SL}_n(\CC)$ is isomorphic to
$\CC^{n-1}$.
The coordinates are the coefficients of the characteristic polynomial
(see \cite[Example~1.2]{Dol04}).
\end{enumerate}
\end{example}

\smallskip

Work of C.~Procesi \cite{Pro76} implies that there exists a finite number of group elements
$\{\gamma_i\mid 1\leq i \leq M\}\subset\Gamma$ such that the image of 
$t\co R_n(\Gamma)\to\CC^M$ given by 
\[t(\rho) = \big(\chi_\rho(\gamma_1),\ldots,\chi_\rho(\gamma_M)\big)\] 
can be identified with the affine algebraic set $X_n(\Gamma)\cong t(R_n(\Gamma))$,
see also \cite[p.~27]{LM85}. This justifies the name \emph{character variety}.

\begin{example} 
\begin{enumerate}
\item Let $F_2$ be the free group on the two generators $x$ and $y$. Then it is possible to show
that $X_2(F_2)\cong \CC^3$ and $t\co R_2(F_2)\xrightarrow{\cong}\CC^3$ given by 
$t(\rho) = \big(\chi_\rho(x),\chi_\rho(y),\chi_\rho(xy)\big)$.
See Goldman's article \cite[Chap.~15]{TeichII} and the article of Gonz\`alez-Acu\~na and 
Montesinos-Amilibia \cite{GM93} for more details.

\item
We obtain $X_3(\mathbb{Z})\cong \CC^{2}$
More precisely,  
$R_3(\mathbb Z)\cong \mathrm{SL}_3(\mathbb C)$ and 
$t\co R_3(\mathbb Z)\xrightarrow{\cong} \mathbb C^2$ is given by 
$t(A) =(\tr(A),\tr(A^{-1}))$.

\item Explicit coordinates for $X_3(F_2)$ are also known: 
$X_3(F_2)$ is isomorphic to a degree 6 affine hyper-surface in $\CC^9$
(see Lawton \cite{Law07}).
\item If $\Gamma$ is a finite group then $X_n(\Gamma)$ is finite for all $n$. 
This follows since $\Gamma$ has up to equivalence only finitely many irreducible representations, and every representation of a finite group is semisimple (see \cite{Serre}).
\end{enumerate}
\end{example}

\subsection{Tangent spaces and group cohomology.} 
The general reference for group cohomology is Brown's book \cite{Bro82}. In order to shorten  notation we will sometimes write $\SL(n)$ and $\mathfrak{sl}(n)$ instead of 
$\SL_n(\CC)$, and $\mathfrak{sl}_n(\CC)$.

The following construction was presented by A.~Weil \cite{Weil64}.
For $\rho\in R_n(\Gamma)$ the Lie algebra $\mathfrak{sl}(n)$ turns into a $\Gamma$-module via $\Ad\circ\rho$, i.e.\
for $X\in\mathfrak{sl}(n)$ and $\gamma\in\Gamma$ we have 
$\gamma\cdot X = \Ad_{\rho(\gamma)} (X) = \rho(\gamma) X \rho(\gamma)^{-1}$.
In what follows this $\Gamma$-module will be denoted by $\sln_{\Ad\rho}$.
We obtain an inclusion 
$T^\mathit{Zar} R_n(\Gamma)\hookrightarrow Z^1(\Gamma, \mathfrak{sl}(n)_{\Ad\rho})$: 
for a smooth family of representations $\rho_t$ with $\rho_0=\rho$ we obtain a map
$u\co\Gamma\to\mathfrak{sl}(n)$ given by 
\begin{equation}\label{eq:integrable}
u(\gamma) = \frac{d \rho_t(\gamma)}{dt}\Big|_{t=0} \rho(\gamma)^{-1}\,.
\end{equation}
The map $u$ verifies: $u(\gamma_1\gamma_2)= u(\gamma_1) + \gamma_1\cdot u(\gamma_2)$
i.e.\ $u\in Z^1(\Gamma, \mathfrak{sl}(n)_{\Ad\rho})$ is a \emph{cocycle} or \emph{derivation} in group cohomology.
If $\rho_t = \Ad_{A_t}\circ\rho$ is contained in $\mathrm{O}(\rho)$ where $A_t$, $A_0=I_n$, is a path of matrices, then the corresponding cocycle is a coboundary  i.e.\ there exists $X\in\mathfrak{sl}(n)$ such that 
$u(\gamma) = (1-\gamma)\cdot X = X - \Ad_{\rho(\gamma)}(X)$. 

\medskip

\paragraph{Attention!} 
The inclusion 
$
T^\mathit{Zar} R_n(\Gamma)\hookrightarrow Z^1(\Gamma, \mathfrak{sl}(n)_{\Ad\rho})
$
might be strict (see Example~\ref{ex:LM}).
More precisely, the space $Z^1(\Gamma;\mathfrak{sl}(n)_{\Ad\rho} )$ 
is the  Zariski tangent space
to the scheme $\mathcal{R}( \Gamma, \mathrm{SL}_n( \mathbb C))$ at $\rho$ (see Section~\ref{subsec:scheme}).
\begin{definition}
Let $\rho\co\Gamma\to\SL(n)$ be a representation.
A derivation $u\in Z^1(\Gamma;\mathfrak{sl}(n)_{\Ad\rho} )$ is called \emph{integrable} if there exists a family of representations $\rho_t\co\Gamma\to\SL(n)$ such that $\rho_0=\rho$ and \eqref{eq:integrable} holds.
\end{definition}

\subsubsection{Detecting smooth points}
The following is a quite  useful observation \cite[p.~iv]{LM85} for detecting smooth points of the representation variety.
In general not every cocycle is integrable and there are different reasons for this.
We have  the following inequalities  
\begin{equation}\label{eq:tangent}
\dim_\rho R_n(\Gamma)\leq\dim T^\mathit{Zar}_\rho R_n(\Gamma)\leq
\dim Z^1(\Gamma, \mathfrak{sl}(n)_{\Ad\rho})
\end{equation}
where $\dim_\rho R_n(\Gamma)$ denotes the local dimension of $R_n(\Gamma)$ at $\rho$ i.e.\
the maximum of the dimensions of the irreducible components of $R_n(\Gamma)$ containing $\rho$.

In what follows, will call $\rho$ a \emph{regular} or \emph{scheme smooth} point of $R_n(\Gamma)$
if the equality 
$\dim_\rho R_n(\Gamma)=\dim Z^1(\Gamma, \mathfrak{sl}(n)_{\Ad\rho})$ holds.
In this case every derivation is integrable, and we obtain the following:
\begin{lemma}[see \protect{\cite[Lemma~2.6]{Heusener-Porti-Suarez2001}}]\label{lem:smoothness}
Let $\rho\in R_n(\Gamma)$ be a representation. If $\rho$ is regular, then $\rho$ is a smooth point of the representation variety $R_n(\Gamma)$, and $\rho$ is contained in a unique component of $R_n(\Gamma)$ of dimension 
$\dim Z^1(\Gamma;\sln_{\Ad\rho})$.
\end{lemma}
\begin{example}\label{ex:central}
Central representations are smooth points of $R_n(\Gamma_k)$.
Let $\Gamma_k$ be a knot group and $\rho_0\in R_n(\Gamma_k)$ be a central  representation i.e.\ 
$\rho_0(\gamma) = \zeta^{\varphi(\gamma)} \mathrm{Id}_n$ where $\zeta\in\mathbb{C}^*$, $\zeta^n=1$.
Then $\mathfrak{sl}(n)$ is a trivial $\Gamma_k$-module and
\[ Z^1(\Gamma, \mathfrak{sl}(n)) = 
H^1(\Gamma, \mathfrak{sl}(n)) = H^1(\Gamma, \mathbb{Z})\otimes \mathfrak{sl}(n)\]
has dimension $n^2-1$. 

On the other hand the surjection $\varphi\co\Gamma_k\to\mathbb{Z}$ induces an injection
$\varphi^*\co R_n(\mathbb{Z})\hookrightarrow R_n(\Gamma_k)$ where $R_n(\mathbb{Z})\cong\mathrm{SL}_n(\CC)$.
Therefore,
$ n^2-1\leq\dim_{\rho_0} R_n(\Gamma_k)\leq \dim Z^1(\Gamma, \mathfrak{sl}(n)) = n^2-1$,
and $\rho_0\in R_n(\Gamma_k)$ is a regular point which is contained in an unique 
$(n^2-1)$-dimensional component of $R_n(\Gamma_k)$ (the component consist of abelian representations).
\end{example}

We give an example where the first inequality of \eqref{eq:tangent} is strict, and the second is an equality. In this case the representation $\rho$ is a \emph{singular point} of the representation variety, but we will see that in our example $\chi_\rho\in X_2(\Gamma )$ is a smooth point.
\begin{example} \label{ex:dyck}
Let $\Gamma = D(3,3,3)=\langle a,b,c\mid a^3,b^3, c^3, abc \rangle\cong
\langle a,b\mid a^3,b^3, (ab)^3 \rangle$ be the van Dyck group. 
We consider the representation $\rho_0\co\Gamma\to\SL(2)$ given by
\[
\rho_0(a) =\rho_0(b) =A=\big(\begin{smallmatrix} \omega & 0\\ 0&\bar\omega \end{smallmatrix}\big)
\] 
where $\omega$ is a primitive third root of unity.

Let $F(a,b)$ denote the free group of rank two and consider the canonical surjection
$\kappa\co F(a,b)\to\Gamma $. We consider $\mathfrak{sl}(2)$ as a $F(a,b)$-module via
$\Ad\circ\rho_0\circ\kappa$.
%
Now, for every $X,Y\in\mathfrak{sl}(2)$ we obtain a cocycle 
$z\co F(a,b) \to \mathfrak{sl}(2)$ such that $z(a)=X$ and $z(b)=Y$.
By using Fox-calculus \cite[Chapter~9]{BZH}, we obtain for $w\in F(a,b)$ 
\[
z(w) = \frac{\partial w}{\partial a}\cdot X + \frac{\partial w}{\partial b}\cdot Y\,.
\]
This cocycle factors through $\kappa$ if and only if $z(a^3)=z(b^3)=z((ab)^3)=0$.
Writing
$X=\big(\begin{smallmatrix} x_{11} & x_{12}\\ x_{21} & -x_{11}\end{smallmatrix}\big)$
and
$Y=\big(\begin{smallmatrix} y_{11} & y_{12}\\ y_{21} & -y_{11}\end{smallmatrix}\big)$ the
 equation $z(a^3)= 0$ gives $0 = (1+a+a^2)\cdot X = X +\Ad_{A}(X) +\Ad^2_{A}(X)$ 
 and hence $x_{11}=0$.
Similar $z(b^3)=0$ gives $y_{11}=0$. The equation 
$z((ab)^3)=0$ gives no further restrictions.
Hence the space of cocycles $Z^1(\Gamma , \mathfrak{sl}(2)_{\Ad\rho_0})$
is $4$-dimensional. The space $B^1(\Gamma , \mathfrak{sl}(2)_{\Ad\rho_0})$ is $2$-dimensional and generated by $b_1$ and $b_2$ which are given by
\[
 b_1\co a,b \mapsto \big(\begin{smallmatrix} 0& 1\\ 0 & 0\end{smallmatrix}\big)
 \text{ and }
  b_2\co a,b \mapsto \big(\begin{smallmatrix} 0& 0\\ 1 & 0\end{smallmatrix}\big)\,.
\]
Two non-principal derivations are given by 
\[
z_1(a)= \big(\begin{smallmatrix} 0& 0\\ 0 & 0\end{smallmatrix}\big),\quad 
z_1(b)= \big(\begin{smallmatrix} 0& 1\\ 0 & 0\end{smallmatrix}\big),
\quad
\text{ and }
\quad
z_2(a)=  \big(\begin{smallmatrix} 0& 0\\ 0 & 0\end{smallmatrix}\big),\quad
z_2(b) = \big(\begin{smallmatrix} 0& 0\\ 1 & 0\end{smallmatrix}\big)\,.
\]

These two derivations are integrable, more precisely the two families $\rho_i(t)\co\Gamma \to\SL(2)$ are given by
$\rho_i(t)(\gamma) = (I_2 + t z_i(\gamma))\rho(\gamma)$, or explicitly by
\[
\rho_1(t) :
\begin{cases}
a & \mapsto \big(\begin{smallmatrix} \omega& 0\\ 0 & \bar\omega\end{smallmatrix}\big),\\[1ex]
b & \mapsto \big(\begin{smallmatrix} 1& t\\ 0 & 1\end{smallmatrix}\big)
                     \big(\begin{smallmatrix} \omega& 0\\ 0 & \bar\omega\end{smallmatrix}\big)=
                     \big(\begin{smallmatrix} \omega& t\bar\omega\\ 0 & \bar\omega\end{smallmatrix}\big),\end{cases}
\qquad
\rho_2(t) :
\begin{cases}
a & \mapsto \big(\begin{smallmatrix} \omega& 0\\ 0 & \bar\omega\end{smallmatrix}\big),\\[1ex]
b & \mapsto \big(\begin{smallmatrix} 1& 0\\ t & 1\end{smallmatrix}\big)
                   \big(\begin{smallmatrix} \omega& 0\\ 0 & \bar\omega\end{smallmatrix}\big)=
                   \big(\begin{smallmatrix} \omega& 0\\ t\omega & \bar\omega\end{smallmatrix}\big)\,.
\end{cases}
\]
It follows that
$\dim T^\mathit{Zar}_{\rho_0} R_n(\Gamma)=\dim Z^1(\Gamma, \mathfrak{sl}(n)_{\Ad\rho_0})$.

Now notice that if $A\in\SL(2)$ verifies 
$A^3=I_2$ and $A\neq I_2$, then 
$A$ is conjugate to $\big(\begin{smallmatrix} \omega & 0\\ 0&\bar\omega \end{smallmatrix}\big)$
where $\omega$ is a third root of unity, $\omega^2+\omega+1=0$.  
Hence we have for all $A\in\SL(2)$:
\[  A^3=I_2 \quad\Longleftrightarrow\quad  A= I_2 \text{ or } \tr A =-1\,.\]

\begin{lemma}\label{lem:triangular}
All representation $\rho\co D(3,3,3)\to\SL(2)$ are reducible.
More precisely, if   $\rho(a)=I_2$ or $\rho(b)=I_2$ or $\rho(ab))=I_2$ is trivial then $\rho$ is conjugate to a diagonal representation. If
$\tr\rho(a)=\tr\rho(b)=\tr\rho(ab)=-1$ then $\chi_\rho=\chi_{\rho_0}$, and $\rho$
is conjugate to an upper/lower triangular representation.
\end{lemma}
\begin{proof}
Let $\rho\co D(3,3,3)\to\SL(2)$ be a representation. Then $\rho$ is trivial if and only if
$\rho(a)=\rho(b)=I_2$.
If $\rho(a)=I_2$ and $\rho(b)\neq I_2$ then up to conjugation we may assume that
$\rho(b) = \big(\begin{smallmatrix} \omega & 0\\ 0&\bar\omega \end{smallmatrix}\big)$. We obtain that $\rho$ is a diagonal representation. A similar argument applies if
$\rho(b)=I_2$ or $\rho(ab)= I_2$.

Now suppose that $\tr\rho(a)=\tr\rho(b)=\tr\rho(ab)=-1$. Up to conjugation we obtain that
$\rho(a)=\big(\begin{smallmatrix} \omega& 0\\ 0 & \bar\omega\end{smallmatrix}\big)$ and
$\rho(b)=\big(\begin{smallmatrix} b_{11}& b_{12}\\ b_{21} & b_{22}\end{smallmatrix}\big)$ where
$b_{11} + b_{22} =-1$ and $b_{11}b_{22}-b_{12}b_{21}=1$. The equation $\tr\rho(ab) =-1$ then implies
$
\big(\begin{smallmatrix}
1 & 1\\ \omega&\bar\omega
\end{smallmatrix}\big) \big(\begin{smallmatrix} b_{11}\\ b_{22}\end{smallmatrix}\big) = 
\big(\begin{smallmatrix} -1\\ -1\end{smallmatrix}\big)$. This system has the unique solution
$
(b_{11}, b_{22})=(\omega,\bar\omega)
$.
Finally, $b_{11}b_{22}-b_{12}b_{21}=1$ implies that $b_{12}b_{21}=0$ and $\rho$ is a triangular representation.
\end{proof}

Notice that a cocycle $z=c_1z_1+c_2 z_2$, $c_1,c_2\in\CC$, is integrable if and only if $c_1c_2=0$ i.e.\ 
only multiples of $z_1$ and multiples of $z_2$ are integrable.

It follows from Lemma~\ref{lem:triangular} that the families $\rho_1(t)$ and $\rho_2(t)$ together form a \emph{slice \'etale} $\mathcal{S}_0$ through the representation $\rho_0$
(see \cite{Leila02} for more details). The slice $\mathcal{S}_0$ is isomorphic to the union of the two coordinate axes in $\CC^2$,
\[
\mathcal{S}_0\cong \CC\times\{0\}\cup\{0\}\times\CC\subset\CC^2,\quad \rho_1(s)\mapsto (s,0) \text{ and } \rho_2(t)\mapsto (0,t)\,.
\]

It follows form \cite[Prop.~2.8]{Leila02} that $H^1(\Gamma ,\mathfrak{sl}(2)_{\Ad\,\rho_0})$ is isomorphic to the tangent space $T_{\rho_0}^\mathit{Zar}\mathcal{S}_0$, and that
$T_{\chi_{\rho_0}}^\mathit{Zar} X_2(\Gamma )$ is isomorphic to
$T_{\chi_{\rho_0}}^\mathit{Zar}\big( \mathcal{S}_0\sslash \mathrm{Stab}(\rho_0)\big)$.
Now, $\mathrm{Stab}(\rho_0)\cong \CC^*$ consits of diagonal matrices,
 and $\lambda\in\CC^*$ acts as follows
\[
\lambda\cdot \rho_1(s) = \rho_1(\lambda^2s) \quad\text{ and }\quad
\lambda\cdot \rho_2(t) = \rho_1(\lambda^{-2}t)
\]
(see Example~\ref{ex:GIT}.\ref{ex:GIT-C^*}). It follows that 
$\mathcal{S}_0\sslash \mathrm{Stab}(\rho_0)\cong\{0\}$ is just a point, and
that $T_{\chi_{\rho_0}}^\mathit{Zar} X_2(\Gamma )$ vanishes.
Notice that all representations $\rho_i(t)$ are equivalent to $\rho_i(1)$. On the other hand
$\rho_1(1)$ and $\rho_2(1)$ are not equivalent. Again,
$\mathrm O(\rho_i(t))$ is not closed, but 
$\mathrm{O}(\rho_0)= \overline{\mathrm O(\rho_1(t))}\cap\overline{\mathrm O(\rho_2(t))}$.
All representations $\rho_0$, $\rho_1(t)$, and $\rho_2(t)$ have the same character.

Notice also that $H^1(\Gamma ,\mathfrak{sl}(2)_{\Ad\,\rho_0})$ is isomorphic to the tangent space of the slice \'etale, and that
$H^1(\Gamma ,\mathfrak{sl}(2)_{\Ad\,\rho_0})\sslash \mathrm{Stab}(\rho)\cong \CC$. This shows that in 
\cite[Theorem~53]{Sik12} the hypothesis scheme smooth can not be omitted.
\end{example}
\begin{remark}
Example~\ref{ex:dyck} can be generalized  to a representation of the fundamental group of the closed $3$-dimensional Seifert fibred manifold $M$ which is an oriented Seifert-bundle over the orbifold $S^2(3,3,3)$. The fundamental group $\pi_1(M)$ is a central extension of $D(3,3,3)$ with presentation
\[
\pi_1M=\langle a,b,c,z\mid a^3=b^3=c^3=abc=z,[a,z],[b,z],[c,z]\rangle\cong \langle a,b,c\mid a^3=b^3=c^3=abc\rangle\,.
\]
It is easy to see that a diagonal representation $\rho_t\co\pi_1M\to\SL(2)$ given by
$\rho_t\co a,b,c\mapsto \big(\begin{smallmatrix} t& 0\\ 0 & t^{-1}\end{smallmatrix}\big)$
is a singular point of the representation variety if and only if $1+t^2+t^4=0$.
\end{remark}

\subsection{The scheme $\mathcal{R}_n( \Gamma)$.}\label{subsec:scheme}
Let $R$ be a commutative and unitary ring.
A \emph{radical ideal} is an ideal $I\subset R$ such that 
$I=\sqrt{I}=\{r\in R \mid r^k\in I\ \hbox{for some positive integer}\ k\}$.
Notice that $I\subset R$ is radical if and only if the quotient ring $R/I$ is reduced i.e.\ 
$R/I$
has no non-zero nilpotent elements. By virtue of Hilbert's Nullstellensatz there is a bijection between algebraic subsets in $\CC^N$ and radical ideals of $\CC[x_1,\ldots,x_N]$ (see \cite{Shaf1,Shaf2}). 
Recall that over $\CC$ a vanishing ideal $\mathcal{I}(V)$ is always radical 
(see Section~\ref{subsubsec:affinealg}).

Now, the ideal generated by the algebraic equations defining the representation
variety may be non-radical (see Example~\ref{ex:LM}). Therefore, one considers the 
underlying \emph{affine scheme} 
$\mathcal{R}_n( \Gamma) := \mathcal{R}( \Gamma, \mathrm{SL}_n( \mathbb C))$ with a possible non-reduced
coordinate ring. 
Weil's construction gives an isomorphism
\[
T^\mathit{Zar}_\rho \mathcal{R}( \Gamma, \mathrm{SL}_n( \mathbb C)) \cong Z^1(\Gamma;\mathfrak{sl}(n)_{\Ad\rho} )\,.
\]
Each
$d\in Z^1(\Gamma;\mathfrak{sl}_n(\mathbb C))$ gives the infinitesimal deformation
 $
\gamma\mapsto  (1+ \varepsilon\, d(\gamma)) \rho(\gamma), \
\forall\gamma\in\Gamma,
 $
 which satisfies the defining equations for 
 $\mathcal{R}_n( \Gamma)$ up to terms in the ideal
$(\varepsilon^2)$ of $\mathbb C[\varepsilon]$, i.e. a Zariski tangent vector
to $\mathcal{R}_n( \Gamma)$ (see \cite[Prop.~2.2]{LM85} and \cite{Leila02}).

\subsubsection{The difference between a scheme and a variety -- heuristics and examples.}
We start with some heuristics. For more details see Shafarevich's book \cite[5.1]{Shaf2}. Here we are only interested in 
\emph{affine schemes} which correspond to rings of the form $R=\CC[x_1,\ldots,x_N]/I$, for an ideal $I\subset\CC[x_1,\ldots,x_N]$.
It may happen that $R=\CC[x_1,\ldots,x_N]/I$ is not reduced. In this case can consider the reduced ring
$R_\mathit{red}=\CC[x_1,\ldots,x_N]/\sqrt{I}$ which is the coordinate ring of the variety
$V=\mathcal{V}(I)\subset \CC^N$. The underlying space of the scheme corresponding to $R$ is $\mathrm{Spec}\,R$ the set of \emph{prime ideals} of $R$. 
Since the kernel of $\pi\co R\twoheadrightarrow R_\mathit{red}$ is the nilradical, it follows that $\pi^*\co \mathrm{Spec}\, R_\mathit{red}\to\mathrm{Spec}\,R$ is a homeomorphism of topological spaces (the two spectra are equipped with the Zariski-topology, see \cite[5.1.2]{Shaf2}). 
Now, the points of $V$ correspond to the maximal ideals of $R_\mathit{red}$ which in turn correspond to the maximal ideals in $R$.

On the other hand, the regular functions on $\mathrm{Spec}\, R_\mathit{red}$ and $\mathrm{Spec}\,R$ are different: 
a non-zero nilpotent element $f\in R$ gives a non-zero function on $\mathrm{Spec}\,R$, but $\pi(f)$ is zero in 
$R_\mathit{red}$. This means that there are non-zero functions on $\mathrm{Spec}\,R$ which take the value zero on every point of $V$. 
These functions may affect the calculation of the tangent space.
One can visualize the scheme corresponding to $R$ as containing some extra \emph{normal material} 
which is actually not tangent to a dimension present in the variety.
\begin{example}
The ring $R=\mathbb{C}[T]/(T^2)$ is not reduced, $R_\mathit{red}=\mathbb{C}[T]/(T)$.
Both rings have only one maximal ideal $(T)\subset R$, and $(0)\subset R_\mathit{red}\cong\CC$.
The zero locus of $(T^2)$ and $(T)$ is the same, it is just the point $\{0\}\in\mathbb{C}$.
The projections $\mathbb{C}[T] \to \mathbb{C}[T]/(T^2)$ and $\mathbb{C}[T] \to \mathbb{C}[T]/(T)$
give inclusions $\mathrm{Spec}\,R \hookrightarrow \CC$, and 
$\mathrm{Spec}\, R_\mathit{red}  \hookrightarrow \CC$. Now, the restriction of a function
$f\in \mathbb{C}[T]$ vanishes on $R_\mathit{red}$ if and only if $f\in (T) \Leftrightarrow f(0)=0$.
On the other hand, the restriction of $f$ onto $\mathrm{Spec}\,R$ vanishes if and only if 
$f\in (T^2)\Leftrightarrow f(0)=0$ and $f'(0)=0$. Hence there are non-zero regular functions
on $\mathrm{Spec}\, R$ which are zero on every point of $\mathrm{Spec}\, R$.
This affect the calculation of the Zariski tangent space:
$$
T^\mathit{Zar}_0 \mathrm{Spec}\, R_\mathit{red}=0,\text{ but }
T^\mathit{Zar}_0 \mathrm{Spec}\, R \cong \mathbb{C}\,.
$$

Notice that $\mathrm{Spec}\, \CC[x]/(x^2)$ appears naturally if we intersect the parabola $\mathcal{V}(y-x^2)$ with the coordinate axis $\mathcal{V}(y)$ in $\CC^2$. See \cite[II.3]{Schemes} for a detailed discussion.
\end{example}

\begin{remark}
There is also an associated character scheme
\[
\mathcal{X}(\Gamma,\mathrm{SL}_n( \mathbb C)) = \mathcal{R}(\Gamma,\mathrm{SL}_n( \mathbb C))\sslash\mathrm{SL}_n( \mathbb C)\,.
\]
\end{remark}

In general the relation between the cohomology group 
$H^1(\Gamma, \mathfrak{sl}(n)_{\Ad\rho})$ and the tangent space
$T^\mathit{Zar}_{\chi_\rho} \mathcal{X}(\Gamma,\mathrm{SL}_n( \mathbb C))$
is more complicate. However, if $\rho$ is an 
irreducible regular representation then we have for the character variety
\[ T^\mathit{Zar}_{\chi_\rho} X_n(\Gamma) \cong 
H^1(\Gamma, \mathfrak{sl}(n)_{\Ad\rho}) \cong 
T^\mathit{Zar}_{\chi_\rho} \mathcal{X}(\Gamma,\mathrm{SL}_n( \mathbb C))\,.\]
(See \cite[Lemma~2.18]{LM85},  and \cite[Section~13]{Sik12} for a generalisation to completely reducible regular representations.)

The next example is a representation $\rho\co\Gamma\to\SL(2)$ such that
$\dim_\rho R_2(\Gamma)=\dim T^\mathit{Zar}_\rho R_2(\Gamma)$ and
$\dim T^\mathit{Zar}_\rho R_2(\Gamma)<\dim Z^1(\Gamma, \mathfrak{sl}(2)_{\Ad\rho})$.
Hence the coordinate ring of the associated scheme has nilpotent elements.

%
\begin{example}\label{ex:LM}
Following Lubotzky and Magid   \cite[pp.~40--43]{LM85} we give an example of a finitely presented group $\Gamma$ and a representation $\rho\co\Gamma\to\SL(2)$ with
non reduced coordinate ring.

For motivation we start with the dihedral group $D_3 = \langle a, s\mid a^3, s^2, sas^{-1} = a^{-1}\rangle$, and a representation $r\co D_3 \to \mathrm{Iso}(\CC)$.
Recall that a transformation $\sigma\in\mathrm{Iso}(\CC)$ is of the form 
\[
\sigma\co z\mapsto \zeta\,z +\alpha \quad\text{ or }\quad \sigma\co z\mapsto \zeta\,\bar z +\alpha
\]
where $\alpha\in\CC$, and $\zeta\in\CC^*$, $|\zeta|=1$, is a complex number of norm $1$. 
A homomorphism
$r\co D_3 \to \mathrm{Iso}(\CC)$ is given by
\[ 
r(a) \co z\mapsto  \omega\,z \quad\text{ and }\quad r(s) \co z\mapsto  \bar z
\]
where $\omega$ is a third root of unity $\omega^2+\omega+1=0$.
The image $r(D_3)$ is contained in $\mathrm{Iso}(\CC)_0 := \{\sigma\in \mathrm{Iso}(\CC)\mid \sigma(0) = 0\}$. Notice also that $\CC\rtimes\mathrm{Iso}(\CC)_0 = \mathrm{Iso}(\CC)$
where $\CC$ is identified with the subgroup of translations.
Let us consider the two translations $\tau_1,\tau_2\co\CC\to \CC$ given by
\[
\tau_1 \co z\mapsto   z + (1+\eta) \quad\text{ and }\quad 
\tau_2 \co z\mapsto  z + (1+\bar\eta) = z + (2 -\eta)
\]
where $\eta$ is a primitive $6$-th root of unity, $\eta^2=\omega$.
An elementary calculation shows that
\begin{alignat*}{2}
r(s)\,\tau_1\,r(s)^{-1} &= \tau_2, \quad& \quad  r(a)\,\tau_1\,r(a)^{-1} &= \tau_2^{-1}\\
r(s)\,\tau_2 \, r(s)^{-1} &= \tau_1, &  r(a)\,\tau_2\,r(a)^{-1} &= \tau_1\tau_2^{-1}\,.
\end{alignat*}
Finally we define the group $\Gamma = (\ZZ\times\ZZ)\rtimes \mathrm{Dic}_3$
where $\mathrm{Dic}_3 = \langle a, s\mid a^6, s^2=a^3, sas^{-1} = a^{-1} \rangle$
is the binary dihedral group of order $12$. The group $\Gamma$ has the following presentation:
\[ 
\langle a,s,t_1,t_2 \mid a^6, s^2=a^3, sas^{-1} = a^{-1}, [t_1,t_2], st_1s^{-1}=t_2, st_2s^{-1}=t_1,
at_1a^{-1}= t_2^{-1}, at_2a^{-1}= t_1t_2^{-1}\rangle\,.
\]
A homomorphism $\rho\co\Gamma\to\SL(2)$ is given by 
\[ 
\rho(a) = \big(\begin{smallmatrix} \eta &0\\0&\bar\eta\end{smallmatrix}\big),\quad
\rho(s) = \big(\begin{smallmatrix} 0 &1\\-1&0\end{smallmatrix}\big), \text { and }
\rho(t_1) = \rho(t_2)=I_2\,.
\]
An elementary but tedious calculation shows that  $z\co\Gamma\to\mathfrak{sl}(2)$ given by
\[ 
z(a) = z(s) = 0, \text { and }
z(t_1) =\big(\begin{smallmatrix} 0 & 1+\eta \\ -(1+\bar\eta)&0\end{smallmatrix}\big),\quad
z(t_2) = \big(\begin{smallmatrix} 0 & 1+\bar\eta \\-(1+\eta)&0\end{smallmatrix}\big)
\]
is a derivation.
The derivation $z$ is non-principal since for each principal derivation
$b\co\Gamma\to\mathfrak{sl}(2)$ we have $b(t_1) = b(t_2)=0$ since
$\rho(t_1)=\rho(t_2)=I_2$ are trivial. Hence, $H^1(\Gamma,\mathfrak{sl}(2)_{\Ad \rho})$ is non-trivial.
More precisely, $H^1(\Gamma,\mathfrak{sl}(2)_{\Ad\rho})\cong\CC$ is generated by the cohomology class of $z$. 

On the other hand, it can be shown directly that each representation of $\Gamma$ into $\SL(2)$ factors through the finite group 
$\Gamma' = \Gamma/ \llangle t_1t_2\rrangle \cong (\ZZ/3\ZZ) \rtimes \mathrm{Dic}_3$.
More generally, this follows also from \cite[Example~2.10]{LM85}.
Therefore, $X_2(\Gamma)$ is finite and $\chi_{\rho}\in X_2(\Gamma)$ is an isolated point.
It follows that the coordinate ring 
$\mathcal{O}(\mathcal{R}_2(\Gamma))$ is non-reduced.

More concretely, we can use SageMath \cite{sage} to compute the ideal $I$ generated by algebraic equations of the $\SL(2)$-representation variety $R_2(\Gamma)\subset\CC^{16}$ \cite{worksheet}. 
It turns out that $\sqrt{I}$ is generated by $I$ and the equations given by the 
relation $t_1t_2=1$. Therefore, we obtain 
$\mathcal{O}(\mathcal{R}_2(\Gamma))_\mathit{red} \cong \mathcal{O}(\mathcal R_2(\Gamma'))\cong  \mathcal{O}( R_2(\Gamma'))$.

If we impose the corresponding relations i.e.\ if we consider 
the representation $\rho'\co\Gamma\to\SL(2)$ given by 
\[ 
\rho'(a) = \big(\begin{smallmatrix} \eta &0\\0&\bar\eta\end{smallmatrix}\big),\quad
\rho'(s) = \big(\begin{smallmatrix} 0 &1\\-1&0\end{smallmatrix}\big), \quad
\rho(t_1) =  \big(\begin{smallmatrix} \omega &0\\0&\bar\omega\end{smallmatrix}\big) \text{ and } 
\rho(t_2)= \big(\begin{smallmatrix} \bar\omega &0\\0&\omega\end{smallmatrix}\big)
\]
then we obtain $H^1(\Gamma,\mathfrak{sl}(2)_{\Ad\rho'})=0$.
\end{example}

\begin{remark}
M.~Kapovich and J.~Millson proved in \cite{KM13} that there are essentially no restrictions on the \emph{local} geometry of representation schemes of 3-manifold groups to $\SL_2(\CC)$. 
\end{remark}

\section{Deformations of representations}

One way to prove that a certain representation $\rho\in R_n(\Gamma)$ is a smooth point of the representation variety is to show that  every cocycle $u\in Z^1(\Gamma;\sln_{\Ad\rho})$ is integrable
(see Lemma~\ref{lem:smoothness}).
In order to do this, we use the classical approach, i.e.\ we first solve the corresponding formal problem, and then apply  a  theorem of Artin \cite{Artin1968}. 

The formal deformations of a representation $\rho\co\Gamma\to \SL_n(\CC)$ are in general determined by an infinite sequence of obstructions (see \cite{Goldman1984, BenAbdelghani2000,Heusener-Porti-Suarez2001}).
In what follows we let
$C^1(\Gamma;\sln_{\Ad\rho}):=\{ c\co\Gamma\to\sln_{\Ad\rho}\}$ denote the $1$-cochains of $\Gamma$ with coefficients in $\sln$ (see \cite[p.59]{Bro82}).

Let $\rho\co\Gamma\to \SL(n)$ be a representation. A formal deformation of $\rho$ is a homomorphism 
$\rho_\infty\co\Gamma\to \SL_n(\CC\llbracket  t \rrbracket)$

\[\rho_\infty(\gamma)=\exp\left(\sum_{i=1}^{\infty}t^iu_i(\gamma)\right)\rho(\gamma)\, ,
\quad u_i\in C^1(\Gamma;\sln) \]
such that $\mathrm{ev}_0\circ\rho_\infty=\rho$. Here 
$\mathrm{ev}_0\co \SL_n(\CC\llbracket  t \rrbracket)\to \SL_n(\CC)$ is the evaluation homomorphism at $t=0$, and 
$\CC\llbracket t \rrbracket$ denotes the ring of formal power series.

 We will say that $\rho_\infty$ is a formal deformation up to  order $k$ of $\rho$ if $\rho_\infty$ is a homomorphism modulo $t^{k+1}$.

An easy calculation gives that $\rho_\infty$ is a homomorphism up to first order if and only if 
$u_1\in Z^1(\Gamma ; \sln_{\Ad\rho})$ is a cocycle. We call a cocycle 
$u_1\in Z^1(\Gamma;\sln_{\Ad\rho})$ \emph{formally integrable} if there is a formal deformation of $\rho$ with leading term $u_1$.

\begin{lemma}
Let $u_1,\ldots,u_k\in C^1(\Gamma;\sln)$ such that
\[\rho_k(\gamma)=\exp\left(\sum_{i=1}^k t^iu_i(\gamma)\right)\rho(\gamma)\]
is a homomorphism into $\SL_n\big(\CC\llbracket t\rrbracket/(t^{k+1})\big)$. 
Then there exists an obstruction class 
$\zeta_{k+1}:=\zeta_{k+1}^{(u_1,\ldots,u_k)}\in H^2(\Gamma, \sln_{\Ad\rho})$ with the following properties:
\begin{enumerate}[(i)]
\item  There is a cochain $u_{k+1}\co\Gamma\to \sln$ such that
\[\rho_{k+1}(\gamma)=\exp\left(\displaystyle\sum_{i=1}^{k+1}t^iu_i(\gamma)\right)\rho(\gamma)\]
is a homomorphism modulo $t^{k+2}$ if and only if $\zeta_{k+1}=0$.
\item  The obstruction $\zeta_{k+1}$ is natural, i.e. if $f\co\Gamma_1\to\Gamma$ is a homomorphism then $f^*\rho_k:=\rho_k\circ f$ is also a homomorphism modulo $t^{k+1}$ and $f^*(\zeta_{k+1}^{(u_1,\ldots,u_k)})=\zeta_{k+1}^{(f^*u_1,\ldots,f^*u_k)}\in 
H^2(\Gamma_1;\sln_{\Ad f^*\rho})$.
\end{enumerate}
\end{lemma}

\begin{proof}
The proof is completely analogous to the proof of 
Proposition~3.1 in \cite{Heusener-Porti-Suarez2001}. 
We replace $\SL(2)$ and $\sln[2]$ by 
$\SL(n)$ and $\sln$ respectively. 
\end{proof}

The following result streamlines the arguments given in \cite{Heusener-Porti2005} and
\cite{AHJ2010}. It is a slight generalization of Proposition~3.3 in \cite{HM14}.

\begin{proposition}\label{prop:smoothpoint}
Let $M$ be a connected, compact, orientable $3$-manifold with toroidal boundary
$\partial M = T_1\cup\cdots\cup T_k$,
and let
$\rho\co\pi_1M\to \SL(n)$ be a representation.

If $\dim H^1(\pi_1M; \sln_{\Ad\rho}) =k(n-1)$ then $\rho$   is a smooth point of the 
$\SL(n)$-representation variety $R_n(\pi_1M)$. Moreover,
$\rho$ is contained in a unique component of dimension
$n^2-1+k(n-1) - \dim H^0(\pi_1M;\sln_{\Ad\rho})$.
\end{proposition}
\begin{proof}
First we will show that the map 
$\iota^*\co H^2(\pi_1 M;\sln_{\Ad\rho})\to H^2(\pi_1 \partial M;\sln_{\Ad\rho})$ induced by the inclusion
$\iota\co \partial M\hookrightarrow M$ is injective. 

Recall that for any CW-complex $X$ with
$\pi_1(X)\cong \pi_1(M)$ and for any $\pi_1 M$-module $A$ 
there are natural morphisms $H^i(\pi_1 M;A)\to H^i( X ;A)$ which are
isomorphisms
for $i=0,1$ and an injection for $i=2$ (see \cite[Lemma 3.3]{Heusener-Porti2005}).
Note also that $T_j\cong S^1\times S^1$ is aspherical and hence
$H^*(\pi_1 T_j;A)\to H^*( T_j ;A)$ is an isomorphism.

For every representation $\varrho\in R_n(\ZZ\oplus\ZZ)$ we have
\begin{equation}\label{eq:boundary} 
\dim H^0(\ZZ\oplus\ZZ;\sln_{\Ad\varrho}) = \frac12 \dim H^1(\ZZ\oplus\ZZ;\sln_{\Ad\varrho})\geq n-1\,,
\end{equation}
and  $\varrho\in R_n(\ZZ\oplus\ZZ)$ is  regular if
and only if equality holds in~\eqref{eq:boundary}. A prove of this statement can be found in the 
of Proposition~3.3 in \cite{HM14}.

Now, the exact cohomology sequence of the pair $(M,\partial M)$ gives 
\begin{multline*}
\to H^1(M,\partial M; \sln_{\Ad\rho})\\ \to
H^1(M; \sln_{\Ad\rho})\xrightarrow{\alpha} H^1(\partial M; \sln_{\Ad\rho}) 
\xrightarrow{\beta} H^2(M,\partial M; \sln_{\Ad\rho})\\
\to H^2(M ; \sln_{\Ad\rho})\xrightarrow{\iota^*}H^2(\partial M ; \sln_{\Ad\rho})\to  
H^3(M,\partial M; \sln_{\Ad\rho})\to 0\,.
\end{multline*}
Poincar\'e-Lefschetz duality implies that $\alpha$ and $\beta$ are dual to each other.
Therefore, we have $\rk\alpha=\rk\beta$, and from the exactness it follows that
$2\rk\alpha = \dim H^1(\partial M; \sln_{\Ad\rho}) $.
Moreover, we have $H^1(\partial M; \sln_{\Ad\rho}) \cong\bigoplus_{j=1}^k H^1(T_j; \sln_{\Ad\iota_j^*\rho})$ where $\iota_j\co T_j\to M$ denotes the inclusion. Equation~\eqref{eq:boundary} implies that $\dim H^0(T_j; \sln_{\Ad\varrho})\geq n-1$ for all $\varrho\in R_n(\pi_1T_j)$. Hence
\begin{align}
k(n-1) & =\dim H^1(M; \sln_{\Ad\rho})\geq \mathrm{rk} (\alpha) =
\frac 1 2 \dim H^1(\partial M; \sln_{\Ad\rho}) \label{eq:poincare}\\
& = \sum_{j=1}^k \frac 1 2 \dim H^1(T_j; \sln_{\Ad\iota_j^*\rho})
= \sum_{j=1}^k\dim H^0(T_j; \sln_{\Ad\iota_j^*\rho}) \geq k(n-1)\,. \notag
\end{align}
Therefore, equality holds everywhere in \eqref{eq:poincare}. This implies that $\alpha$ is injective, hence $\beta$ is surjective, and
\[\iota^*\co H^2(M ; \sln_{\Ad\rho})\to H^2(\partial M ; \sln_{\Ad\rho})\cong \bigoplus_{j=1}^k H^2(T_j; \sln_{\Ad\iota_j^*\rho})\] 
is injective. Moreover, Equation~\eqref{eq:poincare} implies that
$\dim H^0(T_j;\sln_{\Ad\iota_j^*\rho})=n-1$ holds for all $j=1,\ldots,k$, and consequently  
$\iota_j^*\rho=\rho\circ \iota_{j\#} \in R_n(\pi_1T_j)$ is regular. 
We obtain the following commutative diagram:
$$
\begin{CD}
    H^2(M;\sln_{\Ad\rho}) @>{\iota^*}>> H^2(
\partial M;\sln_{\Ad\rho}) \\
    @AAA         @AA{\cong}A \\
    H^2(\pi_1 M;\sln_{\Ad\rho}) @>{\oplus_{j=1}^k\iota_j^*}>>
    \bigoplus_{j=1}^k H^2(\pi_1{T_j};\sln_{\Ad \iota_j^*\rho})\,.
\end{CD} $$

In order to prove that $\rho\in R_n(\pi_1M)$ is regular, we first show that all cocycles in $Z^1(\pi_1 M, \sln_{\Ad\rho})$ are formally integrable. 
We will prove that all obstructions vanish, by using the fact that the obstructions vanish on the boundary.
 Let 
$u_1,\ldots,u_k\co\pi_1M\to\sln$ be given such that 
\[\rho_k(\gamma)=\exp\left(\sum_{i=1}^kt^iu_i(\gamma)\right)\rho(\gamma)\]
 is a homomorphism modulo $t^{k+1}$. Then the restriction
 $\iota_j^*\rho_k\co\pi_1T_j\to \SL_n(\CC\llbracket t\rrbracket)$ is also a formal deformation of order $k$.
Since $\iota_j^*\rho$ is a regular point of the representation variety $R_n(\pi_1T_j)$, the formal implicit function theorem gives that 
$\iota_j^*\rho_k$ extends to a formal deformation of order $k+1$ (see \cite[Lemma~3.7]{Heusener-Porti-Suarez2001}). Therefore, we have that
\[0=\zeta_{k+1}^{(\iota_j^*u_1,\ldots,\iota_j^*u_k)}=\iota_j^*\zeta_{k+1}^{(u_1,\ldots,u_k)}\]
Now, $\oplus_{j=1}^k\iota_j^*$ is injective and the obstruction $\zeta_{k+1}^{(u_1,\ldots,u_k)}$ vanishes.

Hence all cocycles in $Z^1(\Gamma, \sln_{\Ad\rho})$ are formally integrable. By applying Artin's theorem \cite{Artin1968} we obtain from a formal deformation of $\rho$ a convergent deformation (see \cite[Lemma~3.3]{Heusener-Porti-Suarez2001} or \cite[\S~4.2]{BenAbdelghani2000}).

Thus $\rho\in R_n(\pi_1M)$ is a regular point, and
$\dim H^1(\pi_1M;\sln_{\Ad\varrho})=k(n-1)$. The exactness of
\[0\to H^0(\pi_1M;\sln_{\Ad\varrho})\to\sln \to B^1(\pi_1M;\sln_{\Ad\varrho})\to 0,\] 
and the regularity of $\rho\in R_n(\pi_1M)$ imply:
\begin{align*}
\dim_\rho R_n(\pi_1M) &= \dim Z^1(\pi_1M;\sln_{\Ad\varrho}) \\
&= \dim H^1(\pi_1M;\sln_{\Ad\varrho}) + \dim B^1(\pi_1M;\sln_{\Ad\varrho}) \\ &=k(n-1)+n^2-1 - \dim H^0(\pi_1M;\sln_{\Ad\varrho})\,.
\end{align*}
Finally, the proposition follows from Lemma~\ref{lem:smoothness}.
\end{proof}

\begin{definition}\label{def:infinites-regular}
Let $M$ be a connected, compact, orientable $3$-manifold with toroidal boundary
$\partial M = T_1\cup\cdots\cup T_k$.
We  call  a representation $\rho\co\pi_1M\to\mathrm{SL}_n(\CC) $
\emph{infinitesimally regular} if 
 $\dim H^1(\pi_1M ; \mathfrak{sl}(n)_{\Ad\rho})=k(n-1)$.
\end{definition}

\begin{remark}\label{rem:infinites-regular}
It follows from Proposition~\ref{prop:smoothpoint} that
infinitesimally regular representations are regular points on the representation variety.
\end{remark}

\begin{example}\label{ex:diag}
Let $\Gamma_k$ be a knot group and let $D=\mathrm{diag}(\lambda_1,\ldots,\lambda_n)\in\SL(n)$ be a diagonal matrix.
We define the diagonal representation $\rho_D$  by $\rho_D(\gamma)= D^{\varphi(\gamma)}$. 
Now, $\rho_D$ is the direct sum of the one-dimensional representations $\lambda_i^\varphi$, and
the $\Gamma_k$-module $\sln_{\Ad\rho_D}$ decomposes as:
\[
\sln_{\Ad\rho_A} =  \bigoplus_{i\neq j} \CC_{\lambda_i/\lambda_j}\oplus \CC^{n-1} \,.
\]
Now, for all $\alpha\in\CC^*$ we have $H^1(\Gamma_k; \CC_\alpha) =0$ if and only if $\alpha\neq1$ and $\Delta_k(\alpha)\neq 0$
(see \cite[Lemma~2.3]{BenAH15}).
Here, $\Delta_k(t)$ denotes the Alexander polynomial of the knot $k$.
Hence, $\rho_D$ is infinitesimally regular  if and only if $\lambda_i\neq\lambda_j$ for $i\neq j$ and $\Delta_k(\lambda_i/\lambda_j)\neq 0$ for $1\leq i,j\leq n$.
In this case it follows that $\dim H^1(\Gamma_k; \sln_{\Ad\rho_D}) = n-1$, and $\rho_D\in R_n(\Gamma_k)$ is a regular point. The representation  $\rho_D$ is contained in an unique component of dimension
$n^2-1$. This component is exactly the component of abelian representations
$\varphi^*\co R_n(\mathbb{Z})\hookrightarrow R_n(\Gamma_k)$ (see Example~\ref{ex:central}).
\end{example}

\section{Existence of irreducible representations of knot groups}
Let $k\subset S^3$ be a knot, and let $\Gamma_k$ be the knot group.
Given representations of  $\Gamma_k$ into $\SL(2)$ there are several constructions which give higher dimensional representations.
The most obvious is probably the direct sum of two representatios.

\subsection{Deformations of the direct sum of two representations}\label{secHP15}
Starting from two representations  $\alpha\co\Gamma_k\to\mathrm{SL}_a(\CC)$ and  
$\beta\co\Gamma_k\to\mathrm{SL}_b(\CC)$ such that $a+b=n$, we obtain a family of representations
$\rho_\lambda\in R_n(\Gamma_k)$, $\lambda\in\mathbb{C}^*$, by
$\rho_\lambda = (\lambda^{b\varphi}\otimes \alpha)\oplus (\lambda^{-a\varphi}\otimes \beta)\in R_n(\Gamma_k)$ i.e.\ for all $\gamma\in\Gamma_k$ 
\begin{equation}\label{eq:n=gen} 
\rho_\lambda(\gamma) = 
\begin{pmatrix}
\lambda^{b\varphi(\gamma)} \alpha(\gamma) & \mathbf{0}\\
\mathbf{0} & \lambda^{-a\varphi(\gamma)} \beta(\gamma) 
\end{pmatrix}\,.
\end{equation}
Recall that $\lambda^\varphi\co\Gamma_k\to\mathbb{C}^*$ is given by 
$\gamma\mapsto\lambda^{\varphi(\gamma)}$.

Throughout this section we will assume that $\alpha$ and $\beta$ are both irreducible and infinitesimal regular.

The natural question which arises is if 
 $\rho_\lambda$ can be deformed to irreducible representations, and if this would be possible what could we say about the local structure of $X_n(\Gamma_k)$ at $\chi_{\rho_\lambda}$?

\subsubsection{The easiest case}
%
A very special case is $\alpha=\beta\co\Gamma_k\to\mathrm{SL}_1(\CC)=\{1\}$ are trivial.
Then $\rho_\lambda = \lambda^{\varphi}\oplus \lambda^{-\varphi}\in R_2(\Gamma_k)$ i.e.\ for all $\gamma\in\Gamma_k$ 
\begin{equation}\label{eq:n=2} 
\rho_\lambda(\gamma) = 
\begin{pmatrix}
\lambda^{\varphi(\gamma)} & 0\\
0 & \lambda^{-\varphi(\gamma)} 
\end{pmatrix}\,.
\end{equation}

\begin{example}\label{ex:trefoil}
Let us consider the trefoil knot $k=3_1$.
The knot group of the trefoil knot is given by
\[\Gamma_{3_1} = \langle S,T \mid STS=TST\rangle=\langle x,y\mid x^2=y^3\rangle\]
where $x=STS$ and $y=TS$. A meridian is $m=S=x y^{-1}$.
For every irreducible representation $\rho\in R_2(\Gamma_{3_1})$ there exists a unique 
$s\in\mathbb{C}$ such that $\rho\sim\alpha_s$, where
\[
\alpha_s(x) = \begin{pmatrix} i & 0\\ s & -i \end{pmatrix} \text{ and }
\alpha_s(y) = \begin{pmatrix} \eta & \bar\eta-\eta \\ 0 & \bar\eta \end{pmatrix}\,,
\]
and $\eta^2-\eta+1 = 0$ is a primitive $6$-th root of unity. Moreover,  
$\alpha_s$ is irreducible if and only if  $s\neq 0, 2i$
(see \cite[Lemma~9.1]{HP15} for a proof).

Now, if $s=0$ then the one parameter group
$P(t)=\mathrm{diag}(t,t^{-1})$, $t\in\CC^*$, verifies that 
$\lim_{t\to 0}P(t).\alpha_0$ exists, and is the diagonal representation 
$\rho_{\zeta}$ where $\zeta = i\eta$ is a primitive $12$-th root of unity.
If $s=2i$ we can take 
$P(t)=
\big(\begin{smallmatrix} 
t^{-1} & -t^{-1}\\ 
t & 0 
\end{smallmatrix}\big)$, $t\in\CC^*$, and we obtain
$\lim_{t\to 0}P(t).\alpha_{2i} = \rho_{-\zeta}$.
Therefore, the two diagonal representations $\rho_{\pm\zeta}$ 
are limit of irreducible representations. Notice also that $(\pm\zeta)^2=\eta$ is a primitive $6$-th roof of unity and that $\Delta_{3_1} (\eta)=0$.
\end{example}

This examples shows a general phenomena which goes back to work of E.~Klassen
\cite{Klassen91}.
\begin{theorem}\label{thm:Klassen}
If the diagonal representation $\rho_\lambda\in R_2(\Gamma_k)$ can be deformed to irreducible representations then 
$\Delta_k(\lambda^2)=0$.
\end{theorem}
\begin{proof}
In general the function $R_n(\Gamma)\to\ZZ$ given by 
$\rho\mapsto \dim  Z^1(\Gamma,\mathfrak{sl}(n)_{\Ad\rho})$ is upper-semi continuous which means that for every $k\in\ZZ$ the set 
$\{\rho\in R_n(\Gamma) \mid \dim  Z^1(\Gamma,\mathfrak{sl}(n)_{\Ad\rho})\geq k\}$
is closed. Notice that $Z^1(\Gamma,\mathfrak{sl}(n)_{\Ad\rho})$ is the kernel of a linear map which depends algebraically on $\rho$.

Moreover, if  the representation $\rho_\lambda\in R_2(\Gamma_k)$ can be deformed into irreducible representations then 
$\dim Z^1(\Gamma_k,\mathfrak{sl}(2)_{\Ad\rho_\lambda}) \geq 4$ (see \cite[Lemma~5.1]{HP15}). The $\Gamma_k$-module 
$\mathfrak{sl}(2)_{\Ad\rho_\lambda}\cong \mathbb{C}\oplus
\mathbb{C}_{\lambda^2}\oplus \mathbb{C}_{\lambda^{-2}}$ decomposes
into one-dimensional modules (see Example~\ref{ex:diag}).
Now, $H^1(\Gamma_k,\mathbb{C})\cong \mathbb{C}$ and for $\lambda^2\neq 1$ we have
$B^1(\Gamma_k,\mathbb{C}_{\lambda^2})\cong \mathbb{C}$. Hence,
$\dim Z^1(\Gamma_k,\mathfrak{sl}(2)_{\Ad\rho_\lambda}) \geq 4$ implies that
$H^1(\Gamma_k,\mathbb{C}_{\lambda^{\pm2}})\neq0$ or $H^1(\Gamma_k,\mathbb{C}_{\lambda^{-2}})\neq0$.

Finally, $H^1(\Gamma_k,\mathbb{C}_{\lambda^{\pm2}})\neq 0$ and $\lambda^{\pm2}\neq 1$ implies that 
$\Delta_k(\lambda^{\pm2})=0$ (see Example\ref{ex:diag}).
\end{proof}
\begin{remark}
Notice that $\Delta_k(t)\doteq\Delta_k(t^{-1})$ is symmetric and  hence
$H^1(\Gamma_k,\mathbb{C}_{\lambda^{-2}})\neq0$ if and only if
$H^1(\Gamma_k,\mathbb{C}_{\lambda^{2}})\neq0$.
Here $p \doteq q$ means that  $p,q\in\CC[t^{\pm1}]$ are \emph{associated elements}, i.e.\ there exists some unit 
$c\, t^k\in \CC[t^{\pm 1}]$, 
with $ c\in\CC^*$ and $k\in\ZZ$,  such that $p=c\,t^k\,q$.
\end{remark}

In  general, it is still a conjecture that the necessary condition in Theorem~\ref{thm:Klassen} is also sufficient i.e.\ infinitesimal deformation implies deformation. Nevertheless, we have the following result \cite{Heusener-Porti-Suarez2001}:
\begin{theorem}\label{thm:sufficent_n=2}
Let $k\subset S^3$ be a knot and let $\lambda\in\CC^*$.
If $\lambda^2$ is a simple root of $\Delta_k(t)$ then $\rho_\lambda$ is the limit of irreducible representation. 

More precisely, the character $\chi_\lambda$ of $\rho_\lambda$ is contained in exactly two components. One component $Y_2\cong \mathbb{C}$ only contains characters of abelian (diagonal representations), and the second component $X_\lambda$ contains characters of irreducible representations. 
Moreover,
we have $Y_2$ and $X_\lambda$ intersect transversally at $\chi_\rho$, and $\chi_\lambda$ is a smooth point on 
$Y_2$ and  $X_\lambda$.
\end{theorem}

\begin{remark}
Related results, also for other Lie groups are:
Shors \cite{Shors1991}, Frohman--Klassen \cite{FrohmanKlassen1991}, Herald \cite{Herald1997}, Heusener--Kroll \cite{HeusenerKroll1998}, 
Ben Abdelghani \cite{BenAbdelghani2000,BenAbdelghani2010}, Heusener--Porti \cite{Heusener-Porti2005}.
\end{remark}

\subsubsection{The general case.}
Let us go back to the representation
$\rho_\lambda = (\lambda^{b\varphi}\otimes \alpha)\oplus (\lambda^{-a\varphi}\otimes \beta)\in R_n(\Gamma_k)$ given by
Equation~\eqref{eq:n=gen}:
\[ \rho_\lambda(\gamma) = 
\begin{pmatrix}
\lambda^{b\varphi(\gamma)} \alpha(\gamma) & \mathbf{0}\\
\mathbf{0} & \lambda^{-a\varphi(\gamma)} \beta(\gamma) 
\end{pmatrix}\,.
\]

The following generalization of Theorem~\ref{thm:Klassen} was proved in \cite{HP15}:
\begin{theorem}\label{thm:necessary}
Let $\alpha\co\Gamma_k\to\mathrm{SL}_a(\CC)$ and  
$\beta\co\Gamma_k\to\mathrm{SL}_b(\CC)$ be irreducible, $a+b=n$, and assume that $\alpha$ and $\beta$ are infinitesimal regular.
If $\rho_\lambda\in R_n(\Gamma_k)$ is a limit of irreducible representations then 
$\Delta_1^{\alpha\otimes\beta^*}(\lambda^n)= \Delta_1^{\beta\otimes\alpha^*}(\lambda^{-n})=0$.
\end{theorem}

Let us recall some facts about the \emph{twisted Alexander polynomial}.
For more details see \cite{Wada94,Kitano96, KirkLivingston99,Wada2010,HP15}.
Let $V$ be a complex vector space, and $\rho\co\Gamma_k\to \mathrm{GL}(V)$ a representation.
We let $C_\infty\to C_k$ denote the infinite cyclic covering of the knot exterior.
The \emph{twisted Alexander module}  is the $\mathbb{C}[\mathbb{Z}]\cong\mathbb{C}[t^{\pm1}]$-module $H_i(C_\infty,V)$.
A generator $\Delta^\rho_i(t)$ of its order ideal is called the  \emph{twisted Alexander polynomial}
$\Delta_i^\rho(t)\in\mathbb{C}[t^\pm1]$.
Notice that $H_i(C_\infty,V)\cong H_i(C_k,V[\mathbb{Z}])\cong H_i(\Gamma_k,V[\mathbb{Z}])$ where
$V[\mathbb{Z}] = V\otimes_{\mathbb{C}[\Gamma]}\mathbb{C}[\mathbb{Z}]$
is a $\Gamma_k$ module via $\rho\otimes t^\varphi$.

The dual representation
$\rho^*\co\Gamma\to \mathrm{GL}(V^*)$ is given by $\rho^*(\gamma)(f) = f\circ \rho(\gamma)^{-1}$ for $f\in V^* = \mathrm{Hom}(V,\mathbb{C})$ and $\gamma\in\Gamma$. 
In particular, if $\rho\co\Gamma\to \mathrm{GL}(n)$ then $\rho^*(\gamma)=\mbox{}^t\rho(\gamma)^{-1}$ for all $\gamma\in\Gamma_k$.
\begin{lemma}\label{lem:dual}
The representations $\rho$ and $\rho^*$ are equivalent if and only if there exists a $\Gamma$-invariant, non-degenerated bilinear form $V\otimes V\to \mathbb{C}$.
\end{lemma}
\begin{example}
If $\rho\co\Gamma\to\mathrm{O}(n)$ or $\rho\co\Gamma\to\mathrm{SL}_2(\CC)$ 
then $\rho$ and $\rho^*$ are equivalent.
\end{example}

The following theorem is proved in \cite{HP15}:
\begin{theorem}
If $\rho\co\Gamma_k\to \mathrm{GL}(V)$ is a semisimple representation then
$\Delta_i^{\rho^*}(t)\doteq\Delta_i^\rho(t^{-1})$.
\end{theorem}

Now, the proof of Theorem~\ref{thm:necessary} follows the proof of Theorem~\ref{thm:Klassen}.
First, we have to understand the $\Gamma_k$-module $\mathfrak{sl}(n)_{\Ad\rho_\lambda}$.
Let $M_{a, b}(\mathbb C)$ the vector space of $a\times b$ matrices over the complex numbers.
The group $\Gamma_k$ acts on $M_{a, b}(\mathbb C)$ via $\alpha\otimes\beta^*$ i.e.\
for all $\gamma\in\Gamma_k$ and $X\in M_{a, b}(\mathbb C)$ we have
\[
\alpha\otimes\beta^*(\gamma) (X)=
\alpha(\gamma) X \beta(\gamma^{-1})\,.
\]
Similarly, we obtain a representation 
$\beta\otimes\alpha^*\co\Gamma_k\to M_{b,a}(\CC)$.
The proof of the following lemma is given in \cite{HP15}:
\begin{lemma}
If $\alpha\co\Gamma_k\to\mathrm{SL}_a(\CC)$ and  
$\beta\co\Gamma_k\to\mathrm{SL}_b(\CC)$ are irreducible then
the representation $\alpha^*\co\Gamma_k\to\mathrm{SL}_a(\CC)$ is also irreducible.
Moreover, $\alpha\otimes\beta$ and $\beta\otimes\alpha^*$ are semisimple.
\end{lemma}
In what follows we let $\mathcal{M}^+_{t}$ and $\mathcal{M}^-_{t}$ denote the $\Gamma_k$-modules
\[
\mathcal{M}^+_{t} = M_{a, b}(\mathbb C)\otimes\CC[t,t^{-1}] 
\quad\text{ and }\quad
\mathcal{M}^-_{t} = M_{b, a}(\mathbb C)\otimes\CC[t,t^{-1}] 
\]
where $\Gamma_k$ acts via $\alpha\otimes\beta^*\oplus t^\varphi$ and
$\beta\otimes\alpha^*\otimes t^{\varphi}$ repectively.

\begin{corollary}
If $\alpha\co\Gamma_k\to\mathrm{SL}_a(\CC)$ and  
$\beta\co\Gamma_k\to\mathrm{SL}_b(\CC)$ are irreducible then
\[  \Delta_i^{\alpha\otimes\beta^*}(t) \doteq \Delta_i^{\beta\otimes\alpha^*}(t^{-1})\,.\]
\end{corollary} 

Now the $\Gamma_k$-module
$\mathfrak{sl}(n)_{\Ad\rho_\lambda}$ decomposes into a direct sum:
\begin{equation}\label{eq:sln-dec}
\mathfrak{sl}(n)_{\Ad\rho_\lambda}
= \mathfrak{sl}_a(\mathbb
C)_{\Ad\alpha}\oplus \mathfrak{sl}_b(\mathbb C)_{\Ad\beta}\oplus \mathbb C 
\oplus  \mathcal{M}^+_{\lambda^{n}}\oplus  
\mathcal{M}^-_{\lambda^{-n}}\,.
\end{equation}
This can be visualized as
\[
\mathfrak{sl}(n)_{\Ad\rho_\lambda}= 
\begin{pmatrix}
  \mathfrak{sl}(a)_{\Ad\alpha} & \mathcal{M}^+_{\lambda^{n}} \\
   \mathcal{M}^-_{\lambda^{-n}} & \mathfrak{sl}(b)_{\Ad\beta}
\end{pmatrix} \oplus 
\mathbb C 
\begin{pmatrix}
    b \operatorname{Id}_a & 0 \\
    0 & -a \operatorname{Id}_b
\end{pmatrix}.
\]

 For every $\lambda\in\mathbb C^*$ we have a non-degenerate $\Gamma_k$-invariant bilinear form:
$  \Psi\co \mathcal{M}^-_{\lambda^{-n}}\times \mathcal{M}^+_{\lambda^{n}}  \to \mathbb C$
given by $\Psi(Y,X)  \mapsto \mathrm{tr} (YX)$.
As an immediate consequence, we have Poincar\' e and Kronecker dualities:
\begin{eqnarray}
H_i(C;\mathcal{M}^{\pm}_{\lambda^{\pm n}}) & \cong  & 
H_{3-i}(C,\partial C;\mathcal{M}^{\mp}_{\lambda^{\mp n}})^*; \notag\\
H^i(C;\mathcal{M}^{\pm}_{\lambda^{\pm n}}) & \cong  & H^{3-i}(C,\partial C;\mathcal{M}^{\mp}_{\lambda^{\mp n}})^*; \label{eqn:Kronecker}\\
H_i(C;\mathcal{M}^{\pm}_{\lambda^{\pm n}}) & \cong  & H^{i}(C;\mathcal{M}^{\mp}_{\lambda^{\mp n}})^*. \notag
 \end{eqnarray}

\begin{lemma}
If $\alpha\co\Gamma_k\to\mathrm{SL}_a(\CC)$ and  
$\beta\co\Gamma_k\to\mathrm{SL}_b(\CC)$ are irreducible then
$\alpha^*\co\Gamma_k\to\mathrm{SL}_a(\CC)$ is irreducible, and
$\alpha\otimes\beta$ is semisimple.
\end{lemma}

\begin{proof}[Proof of Theorem~\ref{thm:necessary}]
As in the proof of Theorem~\ref{thm:Klassen} it follows from Lemma~5.1 in \cite{HP15} that
if $\rho_\lambda$ is limit of irreducible representations then
\[
\dim Z^1(\Gamma_k,\mathfrak{sl}(2)_{\Ad\rho_\lambda}) \geq n^2+n-2\,.
\]
Now, consider the decomposition \eqref{eq:sln-dec} of $\mathfrak{sl}(n)_{\Ad\rho_\lambda}$.

\smallskip

\paragraph{Claim:} If  $\alpha$ and $\beta$ are infinitesimal regular and irreducible then
\[
\dim H^1(\Gamma_k,\mathcal{M}^+_{\lambda^{n}}) > 
\dim H^0(\Gamma_k,\mathcal{M}^+_{\lambda^{n}})
\quad\text{  or  }\quad
\dim H^1(\Gamma_k,\mathcal{M}^-_{\lambda^{-n}}) > 
\dim H^0(\Gamma_k,\mathcal{M}^-_{\lambda^{-n}})\,.
\]
\begin{proof}[Proof of the Claim.]
For each $\Gamma$-module $\mathfrak m$, we use the formula
\begin{align}
\label{eqn:dimZ1}
\dim Z^1(\Gamma ;  \mathfrak{m}) &=\dim H^1(\Gamma ;  \mathfrak{m})+\dim B^1(\Gamma ;  \mathfrak{m})\notag\\ 
&= \dim H^1(\Gamma ;  \mathfrak{m})+ \dim \mathfrak{m}- \dim H^0(\Gamma ;  \mathfrak{m})\,.
\end{align}
  Ordering the terms as they appear in \eqref{eqn:dimZ1}: 
\begin{eqnarray*}
\dim Z^1(\Gamma ;  \mathfrak{sl}_a(\mathbb C)_{\Ad\,\alpha}) & = & (a-1)+ (a^2-1) - 0, \\
\dim Z^1(\Gamma ;  \mathfrak{sl}_b(\mathbb C)_{\Ad\,\alpha}) & = & (b-1)+ (b^2-1) - 0, \\
\dim Z^1(\Gamma ;  \mathbb C ) & = & 1+ 1 - 1, \\
\dim Z^1(\Gamma ;  \MM^\pm_{\lambda^{\pm n}}) & = &    \dim H^1(\Gamma ;  
\MM^\pm_{\lambda^{\pm n}}) + a\, b  -
\dim H^0(\Gamma ;  \MM^\pm_{\lambda^{\pm n}}).
\end{eqnarray*}
Hence the decomposition \eqref{eq:sln-dec} together with $n^2+n-2\leq \dim Z^1(\Gamma_k,\mathfrak{sl}(2)_{\Ad\rho_\lambda})$
gives:
\begin{multline*}
n^2+n-2 \leq n^2+n-3 +
\big(\dim H^1(\Gamma_k,\mathcal{M}^-_{\lambda^{-n}}) - 
\dim H^0(\Gamma_k,\mathcal{M}^-_{\lambda^{-n}})\big) \\ 
+ \big( \dim H^1(\Gamma_k,\mathcal{M}^+_{\lambda^{n}}) - 
\dim H^0(\Gamma_k,\mathcal{M}^+_{\lambda^{n}})\big)\,.\qedhere
\end{multline*}
\end{proof}
Now, it follows from Kronecker duality~\eqref{eqn:Kronecker} that
\[
\dim H^1(\Gamma_k,\mathcal{M}^-_{\lambda^{- n}}) > 
\dim H^0(\Gamma_k,\mathcal{M}^-_{\lambda^{- n}})
\quad\Leftrightarrow\quad
\dim H_1(\Gamma_k,\mathcal{M}^+_{\lambda^{ n}}) > 
\dim H_0(\Gamma_k,\mathcal{M}^+_{\lambda^{ n}})\,.
\]
The short exact sequence  of $\Gamma_k$-modules 
$0\to\mathcal{M}^+_{t}\xrightarrow{(t-\lambda^n)\cdot}
\mathcal{M}^+_{t}\to \mathcal{M}^+_{\lambda^{ n}}\to 0$ gives a long exact homology sequence \cite[III.\S6]{Bro82}:
   \begin{multline*}
    \ldots\to H_1(\Gamma;\MM_t^{+}) \xrightarrow{(t-\lambda^{-n})\cdot}
    H_1(\Gamma;\MM_t^{+}) \rightarrow \\
    H_1(  \Gamma;\MM^{+}_{\lambda^{n}})  \xrightarrow{\partial\; } 
    H_0(\Gamma;\MM_t^{+}) \xrightarrow{(t-\lambda^{-n})\cdot }
    H_0(\Gamma;\MM_t^{+}) \rightarrow 
    H_0(  \Gamma;\MM^{+}_{\lambda^{n}})  \rightarrow 0\,.
   \end{multline*}
This  implies that
$
\dim H_1(  \Gamma;\MM^{+}_{\lambda^{n}}) \geq \rk(\partial) =
\dim H_0(  \Gamma;\MM^{+}_{\lambda^{n}})
$ with equality if and only if $H_1(  \Gamma;\MM^{+}_{\lambda^{n}})$ has no
$(t-\lambda^n)$-torsion. 
This in turn is equivalent to $\Delta^{\alpha\otimes\beta^*}_1(\lambda^n)\neq0$.

Hence we have:
\[
\Delta^{\alpha\otimes\beta^*}_i(\lambda^n)=0 \Longleftrightarrow
\dim H^1(\Gamma_k,\mathcal{M}^-_{\lambda^{- n}}) > 
\dim H^0(\Gamma_k,\mathcal{M}^-_{\lambda^{- n}})\,.
\]
A similar argument applies if $\dim H^1(\Gamma_k,\mathcal{M}^+_{\lambda^{n}}) > 
\dim H^0(\Gamma_k,\mathcal{M}^+_{\lambda^{n}})$.
\end{proof}
\begin{remark}
Notice that $\Delta^{\alpha\otimes\beta^*}_1(t)\doteq\Delta_1^{\beta\otimes\alpha^*}(t^{-1})$, and hence 
\begin{multline*}
\big(\dim H^1(\Gamma_k,\mathcal{M}^-_{\lambda^{-n}}) >
\dim H^0(\Gamma_k,\mathcal{M}^-_{\lambda^{-n}})\big)
\Leftrightarrow
\Delta^{\alpha\otimes\beta^*}_1(\lambda^n)=0
\Leftrightarrow\\
\Delta^{\beta\otimes\alpha^*}_1(\lambda^{-n})=0
\Leftrightarrow
\big( \dim H^1(\Gamma_k,\mathcal{M}^+_{\lambda^{n}}) > 
\dim H^0(\Gamma_k,\mathcal{M}^+_{\lambda^{n}})\big)\,.
\end{multline*}
\end{remark}
There is a partial converse of Theorem~\ref{thm:necessary} which was proved in \cite{HP15}:
\begin{theorem}
Let $\alpha\co\Gamma_k\to\mathrm{SL}_a(\CC)$ and  
$\beta\co\Gamma_k\to\mathrm{SL}_b(\CC)$ be irreducible, $a+b=n$, and assume that $\alpha$ and $\beta$ are infinitesimal regular.

Assume that  $\Delta_0^{\alpha\otimes\beta^*}(\lambda^n)\neq0$ and that $\lambda ^n$ is a simple root of 
$\Delta_1^{\alpha\otimes\beta^*}(t)$. Then 
$\rho_{\lambda}\in R_n(\Gamma_k)$ can be deformed  to irreducible representations.
Moreover,  the character $\chi_{\lambda}\in X_n(\Gamma_k)$
belongs to precisely two irreducible components $Y$ and $Z$ of
$X_n(\Gamma)$. Both components  $Y$ and $Z$ have dimension
 $n-1$ and meet transversally at $\chi_{\lambda}$ along a subvariety of
dimension {$n-2$}. The component $Y$ contains characters of 
irreducible representations and $Z$
consists only of characters of reducible ones. 
\end{theorem}
\begin{proof}[Sketch of proof]
Use Luna's Slice Theorem, and study the quadratic cone of the representation $\rho_\lambda$ by identifying the second obstruction to integrability. This relies heavily on the hypothesis about the simple root of the Alexander polynomial.
\end{proof}

\subsection{Deformation of reducible metabelian representations}
In this subsection we will consider certain reducible metabelian representations and their deformations. The general assumption will be that $\alpha\in\CC^*$ is a zero of the Alexander polynomial of $k$, and hence $H_1(C_\infty;\CC)$ has a direct summand of the form $\CC[t^{\pm1}] / (t-\alpha)^{n-1}$, $n\in\ZZ$, $n>1$.

Recall that a knot group $\Gamma$ is isomorphic to the semi-direct product
$\Gamma\cong \Gamma'\rtimes\ZZ$.
Every metabelian representation of $\Gamma$ factors through 
the metabelian group $(\Gamma'/\Gamma'')\rtimes \ZZ$. Notice that
$H_1(C_\infty;\CC)\cong \CC\otimes \Gamma'/\Gamma''$. Hence we obtain a homomorphism 
\[
\Gamma\to (\Gamma'/\Gamma'')\rtimes \ZZ \to (\CC\otimes \Gamma'/\Gamma'')\rtimes\ZZ
\to \CC[t^{\pm1}] / (t-\alpha)^{n-1}\rtimes\ZZ\,.
\]
The multiplication on $\CC[t^{\pm1}] / (t-\alpha)^{n-1}\rtimes\ZZ$ is given by
$(p_1,n_1)(p_2,n_2) = (p_1 +t^{n_1} p_2,n_1+n_2)$.

Let $I_{n}\in\SL(n)$ and $N_{n}\in\GLn$ denote  the identity matrix and  
the upper triangular Jordan normal form of a nilpotent matrix of degree $n$ respectively.
 For later use we note the following lemma which follows easily from the Jordan normal form theorem:
\begin{lemma}\label{lem:iso}
Let $\alpha\in\CC^*$ be a nonzero complex number and let $\CC^n$ be the $\CC[t^{\pm1}]$-module 
with the action of $t^k$  given by
\begin{equation}\label{eq:action-J}
 t^k\,\mathbf{a} = \alpha^k\, \mathbf{a}\, J_{n}^k 
  \end{equation}
where $\mathbf a \in \CC^n$ and $J_{n}=I_n+N_n$. 
Then  $\CC^n\cong \CC[t^{\pm1}] / (t-\alpha)^{n}$ as  $\CC[t^{\pm1}]$-modules.
\end{lemma}

There is a direct method to construct a reducible metabelian representation of the group
$\CC[t^{\pm1}] / (t-\alpha)^{n-1} \rtimes\ZZ$ into $\GLn$ (see \cite[Proposition~3.13]{Boden-Friedl2008}).
A direct calculation gives that
\[ (\mathbf{a},0)\mapsto
\begin{pmatrix}
1 & \mathbf{a} \\
\mathbf0 & I_{n-1}
\end{pmatrix} ,\quad (0,1) \mapsto 
\begin{pmatrix}
\alpha & \mathbf{0}  \\
\mathbf0 & J_{n-1}^{-1} 
\end{pmatrix}
\] 
defines a faithful representation  $\CC[t^{\pm1}] / (t-\alpha)^{n-1}\rtimes\ZZ\to\GLn$.

Therefore, we obtain  a reducible, metabelian, non-abelian representation 
$\tilde\metrep\co\Gamma\to\GLn$ if
the Alexander module $H_1(C_\infty,\CC)$ has a direct summand  of the form 
$\CC[t^{\pm1}]\big/(t-\alpha)^s$ with $s\geq n-1 \geq 1$:
\[
\tilde \metrep\co\Gamma \to   
\CC[t^{\pm1}]\big/(t-\alpha)^s\rtimes\ZZ\to
\CC[t^{\pm1}]\big/(t-\alpha)^{n-1}\rtimes\ZZ\to\GLn
\]
given by
\begin{equation}\label{eq:rep-alt-gln}
 \tilde\metrep(\gamma) =
\begin{pmatrix}
1 & \tilde {\mathbf{z}}(\gamma)\\
0& I_{n-1}
\end{pmatrix}
\begin{pmatrix}
\alpha^{\varphi(\gamma)} & 0\\
0& J^{-\varphi(\gamma)}_{n-1}
\end{pmatrix}.
\end{equation}
It is easy to see that a map $\tilde\varrho\co\Gamma\to\GLn$ given by \eqref{eq:rep-alt-gln} is a homomorphism if and only if 
$\tilde{\mathbf{z}}\co\Gamma\to \CC^{n-1}$ is a cocycle 
i.e.\ for all $\gamma_1,\gamma_2\in\Gamma$ we have
\begin{equation} \label{eq:tildez}
\tilde{\mathbf{z}}(\gamma_1\gamma_2) = \tilde{\mathbf{z}}(\gamma_1)
+\alpha^{\varphi(\gamma_1)} \tilde{\mathbf{z}}(\gamma_2) J_{n-1}^{\varphi(\gamma_1)}\,.
\end{equation}
The unipotent matrices 
$J_{n}$ and $J_{n}^{-1}$ are similar: a direct calculation shows that $P_n J_n P_n^{-1} =J_n^{-1} $ where $P_n=(p_{ij})$, $p_{ij}=(-1)^j {j\choose i}$.
The matrix $P_n$ is upper triangular with $\pm1$ in the diagonal and $P_n^2$ is the identity matrix, and therefore $P_n=P_n^{-1}$.

Hence $\tilde\metrep$ is conjugate to a representation 
$\metrep\co\Gamma\to\GLn$ given by
\begin{equation}\label{eq:rep-gln}
\metrep(\gamma) =
\begin{pmatrix}
\alpha^{h(\gamma)} &  z(\gamma)\\
0& J^{h(\gamma)}_{n-1}
\end{pmatrix}
=\begin{pmatrix}
\alpha^{h(\gamma)}&z_1(\gamma)&z_2(\gamma)&\ldots&z_{n-1}(\gamma)\\
0&1&h_1(\gamma)&\ldots&h_{n-2}(\gamma)\\
\vdots&\ddots&\ddots&\ddots&\vdots\\
\vdots&&\ddots&1&h_1(\gamma)\\
0&\ldots&\ldots&0&1
\end{pmatrix}
\end{equation}
where $\mathbf{z}=(z_1,\ldots,z_{n-1})\co\Gamma\to\CC^{n-1}$ satisfies
\[
\mathbf{z}(\gamma_1\gamma_2) =
\alpha^{h(\gamma_1)}\mathbf{z}(\gamma_2) + \mathbf{z}(\gamma_1)J_{n-1}^{h(\gamma_2)}\,.
\]
It follows directly that $\mathbf{z}(\gamma)=\tilde{\mathbf{z}}(\gamma) P_{n-1} J_{n-1}^{h(\gamma)}$ and in particular $z_1= - \tilde z_1$.

We choose an $n$-th root $\lambda$ of $\alpha=\lambda^n$ and we define a reducible metabelian representation 
$\metrep_\lambda\co\Gamma\to\SL(n)$ by
\begin{equation}\label{eq:rep-sln}
\metrep_\lambda(\gamma) = \lambda^{- \varphi(\gamma)} \metrep(\gamma).
\end{equation}

The following theorem 
generalizes the results of \cite{AHJ2010} where the case $n=3$ was investigated. It also applies in the case $n=2$ which was studied in
\cite{BenAbdelghani2000} and~\cite[Theorem~1.1]{Heusener-Porti-Suarez2001}. 
\begin{theorem}\label{thm:mainthm} Let $k$ be a knot in the $3$-sphere $S^3$.
If the $(t-\alpha)$-torsion $\tau_\alpha$ of the Alexander module 
$H_1(C;\CC[t^{\pm1}])$
is cyclic of the form $\CC[t^{\pm1}]\big/(t-\alpha)^{n-1}$, $n\geq 2$, 
then for each $\lambda \in \CC^*$ such that $\lambda^n=\alpha$ there exists a reducible metabelian representation 
$\metrep_\lambda\co\Gamma_k\to\SL(n)$.
Moreover, the representation $\metrep_\lambda$ is a smooth point of the representation variety $R_{n}(\Gamma)$.
It is contained in a unique $(n^2+n-2)$-dimensional component $R_{\metrep_\lambda}$ of $R_n(\Gamma)$
which contains irreducible non-metabelian representations which deform $\metrep_\lambda$. 
\end{theorem}
The main part of the proof of this theorem is a cohomological calculation \cite{BenAH15}:
for the representation $\metrep_\lambda\co\Gamma\to\SL(n)$ we have
$H^0(\Gamma;\sln_{\Ad\circ\metrep_\lambda})=0$ and
\[
\dim H^1(\Gamma;\sln_{\Ad\circ\metrep_\lambda})=\dim H^2(\Gamma;\sln_{\Ad\circ\metrep_\lambda})=n-1\,.
\]
Then we apply Proposition~\ref{prop:smoothpoint}.

\begin{remark}
Let $\rho_\lambda\co\Gamma\to \SL(n)$ be the diagonal representation given by 
$\rho_\lambda(m) =  \mathrm{diag}(\lambda^{n-1},\lambda^{-1},\ldots,\lambda^{-1})\in\SL(n)$ where $m$ is a meridian of $k$. 
The orbit $\mathcal O (\rho_\lambda)$ of
$\rho_\lambda$ under the action of conjugation of $\SL(n)$ is contained in the closure 
$\overline{\mathcal O (\metrep_\lambda)}$. Hence $\metrep_\lambda$ and $\rho_\lambda$ project to the same point $\chi_\lambda$ of the variety of characters 
$X_n(\Gamma_k)= R_n(\Gamma_k)\sslash \SL(n)$.

It would be natural to study the local picture of the variety of characters
$X_n(\Gamma_k)= R_n(\Gamma_k)\sslash \SL(n)$ at $\chi_\lambda$ as done in 
\cite[\S\ 8]{Heusener-Porti2005}. Unfortunately, there are much more technical difficulties since in this case the quadratic cone $Q(\rho_\lambda)$ coincides with the Zariski tangent space 
$Z^1(\Gamma; \sln_{\Ad \rho_\lambda})$. Therefore the third obstruction has to be considered. 
\end{remark}

\subsection{The irreducible representation $r_n\co\SL(2)\to\SL(n)$}
It is interesting to study the behavior of  representations
$\rho\in R_2(\Gamma)$  under the composition with
the $n$-dimensional, irreducible, rational representation 
$r_n\co\SL(2)\to\SL(n)$. The representation $r_n$ is equivalent to $(n-1)$-fold symmetric power 
$\mathrm{Sym}^{n-1}$ of the standard representation (see \cite{Springer1977,Fulton-Harris} and \cite{HM14} for more details). In particular, $r_1$ is trivial, $r_2$ is equivalent to the standard representation, and
$r_3$ is equivalent to $\Ad\co\SL(2)\to\mathrm{O}(\sln[2])\subset \SL(3)$. If $k$ is odd then
$r_k$ is not injective since it factors trough the projection $\SL(2)\to\mathrm{PSL}(2)$.
W.~M\"uller \cite{Muller} studied the case where $\rho\co\pi_1(M)\to\SL(2)$ is the lift of the holonomy representation of a compact hyperbolic manifold. This study was extended by
P.~Menal-Ferrer and J.~Porti \cite{Menal-FerrerPorti2012,Menal-FerrerPorti2014} to the case of
non-compact finite volume hyperbolic manifolds. 
(For more details see Section~\ref{subsec:distComp}.)

In \cite{HM14} the authors studied the case related to Theorem~\ref{thm:sufficent_n=2}.
Let $\Gamma_k$ be a knot group. We define $\rho_{n,\lambda}\co\Gamma_k\to\SL(n)$ by
$\rho_{n,\lambda} := r_n\circ \rho_\lambda$ where $\rho_\lambda$
is given by Equation~\eqref{eq:n=2}. 

\begin{proposition}\label{prop:irred-exist}
Let $k\subset S^3$ be a knot, and  assume that
$\rho_0\co\Gamma_k\to\SL(2)$ is irreducible. 
Then $R_n(\Gamma_k)$ contains irreducible representations.
\end{proposition}

\begin{proof}
It was proved by Thurston that there is  at least a $4$-dimensional irreducible component 
$R_0\subset R_2(\Gamma_k)$ which contains the irreducible representation 
$\rho_0$ (see \cite[3.2.1]{CS83}).

Let $\Gamma$ be a discrete group and let $\rho\co\Gamma\to\SL(2)$ be an irreducible representation. By virtue of Burnside's Theorem on matrix algebras,
being irreducible is an open property for representations in $R_n(\Gamma)$.
If the image $\rho(\Gamma)\subset\SL(2)$ is Zariski-dense then the representation 
$\rho_n := r_n\circ\rho\in R_n(\Gamma)$ is irreducible. 
In order to prove the proposition we will show that there is a neighborhood 
$U=U(\rho_0)\subset R_0\subset R_2(\Gamma_k)$ such that
$\rho(\Gamma)\subset\SL(2)$ is Zariski-dense for each irreducible $\rho\in U$.

Let now $\rho\co\Gamma_k\to\SL(2)$ be any irreducible representation and let 
$G\subset\SL(2)$ denote the Zariski-closure of $\rho(\Gamma_k)$. 
Suppose that $G\neq\SL(2)$.
Since $\rho$ is irreducible it follows that $G$ is, up to conjugation, not a subgroup of upper-triangular matrices of $\SL(2)$. 
Then by \cite[Sec.~1.4]{Kovacic1986} and \cite[Theorem~4.12]{Kaplansky1957}
there are, up to conjugation, only two cases left:
\begin{itemize}
\item $G$ is a subgroup of the infinite dihedral group
\[ D_\infty =\Big\{ \big(\begin{smallmatrix} \alpha & 0\\ 0 &\alpha^{-1}\end{smallmatrix}\big)
\,\big|\, \alpha \in\CC^*\Big\} \cup
\Big\{ \big(\begin{smallmatrix} 0 & \alpha\\ -\alpha^{-1}&0\end{smallmatrix}\big)
\,\big|\, \alpha \in\CC^*\Big\}\,.\]

\item $G$ is one of the groups $A_4^{\SL(2)}$ (the tetrahedral group), 
$S_4^{\SL(2)}$ (the octahedral group) or $A_5^{\SL(2)}$ (the icosahedral group). 
These groups are the preimages  in $\SL(2)$ of the subgroups 
$A_4$, $S_4$, $A_5 \subset \mathrm{PSL}(2,\CC)$.
\end{itemize}

By a result of E.~Klassen \cite[Theorem~10]{Klassen91}
there are up to conjugation only finitely many irreducible representations of a knot group into $D_\infty$.
Moreover, the orbit of each of those irreducible representation is $3$-dimensional.
Therefore, there exists a Zariski-open subset $U\subset R_0$ which does not contain representations of $\Gamma_k$ into $D_\infty$.

For the second case  there are up to conjugation only finitely many irreducible representations of $\Gamma_k$ onto the subgroups
$A_4^{\SL(2)}$, $S_4^{\SL(2)}$ and $A_5^{\SL(2)}$. As in the dihedral case these finitely many orbits are closed and $3$-dimensional. Hence all the irreducible $\rho\in R_0$ such that $r_n\circ\rho$ is reducible 
are contained in a Zariski-closed subset of $R_0$. 
Hence generically $\rho_n = r_n\circ\rho$ is irreducible for $\rho\in R_0$.
\end{proof}

\begin{remark}
Recall that a finite group has only finitely many irreducible representations 
(see \cite{Serre,Fulton-Harris}). Hence, the restriction of $r_n$ to the groups 
$A_4^{\SL(2)}$, $S_4^{\SL(2)}$ and $A_5^{\SL(2)}$ is reducible, for all but finitely many 
$n\in\mathbb{N}$.
\end{remark}

Let $k\subset S^3$ be a knot, and let $\lambda^2\in\CC$ a simple root of $\Delta_k(t)$.
We let $R_\lambda\subset R_n(\Gamma_k)$ denote the 4-dimensional component which maps onto the component $X_\lambda\subset X_2(\Gamma_k)$ under $t\co R_n(\Gamma)\to X_n(\Gamma)$ 
(see Theorem~\ref{thm:sufficent_n=2}). 
We obtain:
\begin{corollary}\label{cor:irred-exist}
Let $k\subset S^3$ be a knot, and $\lambda^2\in\CC$ a simple root of $\Delta_k(t)$.
Then the diagonal representation 
$\rho_{\lambda,n}=r_n\circ \rho_{\lambda}\co\Gamma_k\to \SL(n)$ is the limit of irreducible representations in $R_n(\Gamma_k)$. More precisely, generically a representation 
$\rho_n = r_n\circ \rho$, $\rho\in R_\lambda$, is irreducible.
\end{corollary}
Corollary~\ref{cor:irred-exist} can be made more precise (see \cite{HM14}):
\begin{theorem}\label{thm:smoothR_n}
If $\lambda^2$ is a simple root of $\Delta_k(t)$ and if
 $\Delta_k(\lambda^{2i})\ne0$ for all $2\leq i \leq n-1$ then the reducible diagonal representation
$\rho_{\lambda,n} =r_n\circ \rho_\lambda$ is a limit of irreducible representations.
More precisely, there is a unique $(n+2)(n-1)$-dimensional component 
$R_{\lambda,n}\subset R_n(\Gamma_k)$ which contains $\rho_{\lambda,n}$ and irreducible representations.
 \end{theorem}

\begin{remark} Under the assumptions of Corollary~\ref{cor:irred-exist} it is possible to study
the tangent cone of $R_n(\Gamma_k)$ at $\rho_{\lambda,n}$, and thereby to determine the local structure of $R_n(\Gamma)$. There are
$2^{n-1}$ \emph{branches} of various dimensions of $R_n(\Gamma_k)$ passing through 
$\rho_\lambda$. Nevertheless, only the component
$R_{\lambda,n}$ contains irreducible representations. This will be studied in a forthcoming paper.
\end{remark}

\section{The global structure of character varieties of knot groups}
Not much is known about the global structure of the character varieties of knot groups.
In this section we will present some facts and some examples.

\begin{example}[Diagonal representations]
The characters of diagonal representations of a knot group $\Gamma_k$ form an algebraic component of $X_n(\Gamma_k)$.
A representation $\rho\co\Gamma_k\to\SL(n)$ which is the direct sum of one-dimensional representations is equivalent to a diagonal representation.
The image of a diagonal representation is abelian. 
Hence it factors through $\varphi\colon\Gamma_k\to\ZZ$.
Therefore, the characters of diagonal representations coincide with the characters 
$X_n(\ZZ)\hookrightarrow X_n(\Gamma_k)$.
Recall that $X_n(\ZZ)\cong\CC^{n-1}$. 
\end{example}

\subsection{The distinguished components for hyperbolic knots}\label{subsec:distComp}
Let $k\subset S^3$ be a hyperbolic knot i.e.\ $S^3\smallsetminus k$ has a hyperbolic metric of finite volume. Then there exists, up to complex conjugation,
 a unique one-dimensional component
$X_0\subset X(\Gamma_k,\mathrm{PSL}_2(\mathbf{C}))$ 
which contains the character of the holonomy representation (see \cite[Theorem 8.44]{Kapovich}).
The holonomy representation lifts to a representation 
$\rho\colon\Gamma_k\to\SL(2)$ (not unique) since $H^2(\Gamma_k;\ZZ/2\ZZ)=0$. 
By composing any  
lift with the rational, irreducible, $r$-dimensional representation
$\mathrm{Sym}^{r-1}\co\SL(2)\to\SL(r)$ we obtain an irreducible representation
$\rho_r\co\Gamma_k\to\SL(r)$.
It follows from work of Menal-Ferrer and Porti \cite{Menal-FerrerPorti2012} that 
$\chi_{\rho_r}\in X_r(\Gamma_k)$ is a scheme smooth point contained in a unique 
$(r-1)$-dimensional 
component of $X_r(\Gamma_k)$. 
We will call such a component a \emph{distinguished} component of $X_r(\Gamma_k)$.
For \emph{odd} $r$, as $\mathrm{Sym}^{r-1}\co\SL(2)\to\SL(r)$ factors through $\mathrm{PSL}(2)$, there is a unique distinguished component  
in $X_r(\Gamma)$ up to complex conjugation. 

\subsection{Examples}
The aim of this subsection is to describe the components of the $\SL(3)$-character varieties of the trefoil knot and the figure eight knot, see \cite{HP15,HMP15}.

\subsubsection{Irreducible $\SL(3)$-representations of the trefoil knot group}
Let $k\subset S^3$ be the trefoil knot and $\Gamma=\Gamma_{3_1}$. 
We use the presentation
$$
\Gamma\cong \langle x,y\mid x^2=y^3\rangle\,.
$$
The center of $\Gamma$ is the cyclic group generated by $z=x^2=y^3$. 
The abelianization map $\varphi\co\Gamma\to\ZZ$ satisfies $\varphi(x)=3$, $\varphi(y)=2$, and a meridian of the trefoil is given by $m=xy^{-1}$. Let $\omega$ denote a primitive third root of unity,
$\omega^2+\omega+1=0$. 

For a given representation $\rho\in R_3(\Gamma)$ we put
\[ \rho(x)=A \text{ and }\rho(y)=B\,. \]
If $\rho$ is irreducible then it follows from Schur's Lemma that
the matrix $A^2=B^3\in \{\mathrm{id}_3,\omega\, \mathrm{id}_3, \omega^2\,\mathrm{id}_3\}$ is a central element of 
$\mathrm{SL}(3)$.

\begin{lemma}\label{lem:a^2=b^3}
If $\rho\co\Gamma\to\SL(3)$ is irreducible then
 $A^2=B^3=\mathrm{id}_3$.
\end{lemma}

\begin{proof} The matrix $A$ has an eigenvalue of multiplicity two, and hence $A$ has a two-dimensional eigenspace. Therefore, $B$ has only one-dimensional eigenspaces, otherwise $\rho$ would not be irreducible. This implies that $B$ has three different eigenvalues:
$\lambda$, $\lambda\omega$, $\lambda\omega^2$ where $\lambda^3\in\{1,\omega,\omega^2\}$.
We obtain $\det(B)=1=\lambda^3$. Therefore $B^3=\mathrm{Id}_3$.
\end{proof}

Lemma~\ref{lem:a^2=b^3} implies that the matrices $A$ and $B$ are conjugate
to
$$
A\sim \big(\begin{smallmatrix}
    1  & &\\
    & -1 & \\
    & & -1
\end{smallmatrix}\big)
\quad\textrm{ and }\quad
B\sim \big(\begin{smallmatrix}
    1  & &\\
    & \omega & \\
    & & \omega^2
\end{smallmatrix}\big)\,.
$$
The corresponding eigenspaces are the plane  $E_A(-1)$, and the lines $E_A(1)$, 
$E_B(1)$, $E_B(\omega)$, and  $E_B(\omega^2)$.

Now, these eigenspaces determine the representation completely, as they determine the matrices 
$A$ and $B$, that have fixed eigenvalues. 
Of course we have $E_A(1)\cap E_A(-1)=0$ 
and $E_B(1)$, $E_B(\omega)$, and  $E_B(\omega^2)$ are also in general position.
Since $\rho$ is irreducible, the five eigenspaces are in general position.
For instance  $E_A(1)\cap (E_B(1)\oplus E_B(\omega))=0$, because otherwise
$  E_B(1)\oplus E_B(\omega)=  E_A(1)\oplus (E_A(-1)\cap (E_B(1)\oplus
E_B(\omega)))$ would be a proper invariant subspace.

We now give a parametrization of the conjugacy classes of the irreducible representations.
The invariant lines correspond to fixed points in
the projective plane $\mathbb{P}^2$, and $E_A(-1)$ determines a projective line. 
\begin{itemize}
\item The first normalization: the line $E_A(-1)$ corresponds to the line at infinity:
\[ \mathbb{P}^1=E_A(-1)= \langle [0:1:0], [0:0:1]\rangle \]
\end{itemize} 
The four invariant lines $E_A(1)$, 
$E_B(1)$, $E_B(\omega)$, and  $E_B(\omega^2)$ are points in the
affine plane $\mathbb C^2=\mathbb{P}^2\setminus\mathbb{P}^1$. They are in general position. 
\begin{itemize} 
\item We fix the three fixed points of $B$, corresponding to the following affine frame. 
\[ E_B(1) = [ 1:0:0],\ E_B(\omega) = [ 1:1:0], \text{ and }E_B(\omega^2) =
[1:0:1]. \]

\item The fourth point (the line  $E_A(1)$) is a point in 
$\mathbb C^2$ which does not lie in the affine lines spanned  by any two of the fixed points of $B$:
$E_A(1) = [2:s:t]$ where $s\neq 0$, $t\neq 0$, or $s+t \neq 2$
\end{itemize}
This gives rise to the subvariety 
$\{\rho_{s,t}\in R(\Gamma,\mathrm{SL}(3))\mid (s,t)\in\mathbf{C}^2\}$, where  
\begin{equation*}
\label{eq:rhost}
\rho_{s,t}(x)=\begin{pmatrix}
           1 & 0 &0\\
           s & -1 & 0\\
           t & 0 &-1
          \end{pmatrix}
\quad\textrm{ and }\quad
\rho_{s,t}(y)=\begin{pmatrix}
           1 & \omega -1 & \omega^2 -1\\
           0 & \omega & 0\\
           0 & 0 & \omega^2
          \end{pmatrix}\,.
\end{equation*}

We obtain the following lemma:
\begin{lemma} \label{lem:(s,t)-irreductible} Every irreducible representation
$\rho\co\Gamma_{3_1}\to\SL(3)$ is equivalent to exactly one representation 
$\rho_{s,t}$. Moreover, $\rho_{s,t}$ is reducible
if and only if $(s,t)$ is contained in one of the three affine lines given by
$s=0$, $t=0$, and $s+t =2$. If $(s,t)\in\{(0,0),(0,2),(2,0)\}$ is the intersection point of two of those lines then
$\rho_{s,t}$ fixes a complete flag, and has the character of a diagonal representation.
\end{lemma}

The following theorem follows from the above considerations 
(see \cite[Theorem 9.10]{HP15} for more details). 
We let $R_n^\mathit{irr}(\Gamma)\subset R_n(\Gamma)$ denote the Zariski-open subset of irreducible representation
\begin{theorem}
The GIT quotient
$X=\overline{R_3^\mathit{irr}(\Gamma)}\sslash\mathrm{SL}(3)$
of the trefoil knot group $\Gamma$ is isomorphic to
$\mathbb{C}^2$.
Moreover, the Zariski open subset $R_3^\mathit{irr}(\Gamma)$ is
$\mathrm{SL}(3)$-invariant and its
GIT quotient is isomorphic to the complement of three affine lines in general
position in $\mathbb{C}^2$.
\end{theorem}

\begin{remark}
The same arguments as above apply to torus knots $T(p, 2)$, $p$ odd, to prove that the variety of irreducible
$\mathrm{SL}_3(\mathbb{C})$-characters 
consist of $(p-1)(p-2)/2$ disjoint components isomorphic to $\mathbb C^2$, 
$(p-1)/2$ components of characters of partial reducible representations, 
and the component of characters of diagonal representations. 
\smallskip

In general, the $\SL(3)$-character variety for torus knots  was studied by Mu\~noz and Porti	
\cite{MunPor}. In the general case $T(p,q)$, $p,q>2$ there are 4-dimensional components in $X_3(\Gamma_{T(p,q)})$ corresponding to the configuration of 6 points in the projective plane.
\end{remark}

\subsubsection{The $\SL(3)$-character variety of the figure eight knot}
Let $\Gamma=\Gamma_{4_1}$ be the group of the figure eight knot.
The figure eight knot has genus one, and its complement fibres over the circle \cite{BZH}.
Hence the commutator group of $\Gamma$ is a free group of rank two, and a presentation is given by
\[ 
\Gamma \cong  \langle t,a,b\mid tat^{-1} = ab,\ tbt^{-1} = bab\rangle\,.
\]
A peripheral system is given by $(m,\ell)= (t,[a,b])$.
The amphicheirality of the figure eight knot implies that its group has an automorphism $h\co\Gamma\to\Gamma$ which maps the peripheral system $(m,\ell)$
to $(m^{-1},\ell)$ up to conjugation. Such an automorphism is explicitly given by 
%
\[
h(t) =  t a^{-1} t^{-1} a t^{-1} \sim t^{-1},\quad
h(a) =  a^{-1} t a  b^{-1}  a^{-1} t^{-1} a \sim b^{-1},\quad
h(b) =  a^{-1} t a  t^{-1} a\sim a 
\]
Notice that  we obtain 
$h(m) = t a^{-1}  m^{-1}   t^{-1} a$ and
$h(\ell) = h([a,b]) = a^{-1} t a   [b^{-1},a]   a^{-1} t^{-1} a$.
The relation $t^{-1} a^{-1} t = b a^{-2}$ gives that
the peripheral system $(h(m),h(\ell))$ is conjugated to $(m^{-1},\ell)$ as desired.

The structure of the $\SL(3)$-character variety of the figure eight knot had been studied in detail in \cite{HMP15}, see also \cite{Falbel}. The character variety $X_3(\Gamma_{4_1})$ has 5 components:
\begin{itemize}
\item the component containing the characters of abelian representations;
\item one component containing the characters of the representations 
$\rho_\lambda=\alpha\otimes\lambda^\varphi\oplus\lambda^{-2\varphi}$ where 
$\alpha\in R_2(\Gamma_{4_1})$ is irreducible (compare Equation~\ref{eq:n=gen} with $\beta$ trivial);
\item three components $V_0$, $V_1$ and $V_2$ containing characters of irreducible representations. The component $V_0$ is the distinguished component 
(see Section~\ref{subsec:distComp}). The two other components which come from a surjection 
$\Gamma_{4_1}\twoheadrightarrow D(3,3,4)$ onto a triangle group. 
\end{itemize}

Let us describe the components $V_1$ and $V_2$ without going too much into the technical details.
An epimorphism $\phi\co\Gamma\twoheadrightarrow D(3,3,4)=\langle k,l \mid l^3,\ k^3,\ (kl)^4 \rangle$
 is given by
\[
\phi(a)= k^{-1} l^{-1} k l , \quad \phi(b) = k l \quad \text{ and }\quad  \phi(t) = k l k\,.
\]
It satisfies $\phi(b)^4=1$ and 
$\phi(m^3 \ell) =1$. Notice that the surjection $\phi$ induces an injection
 \[
 \phi^*\co X_3(D(3,3,4) )\hookrightarrow X_3(\Gamma)\,.
 \]

\begin{remark} The surjection $\phi\co\Gamma\twoheadrightarrow D(3,3,4)$ is related to an exceptional  Dehn filling
on the figure-eight knot $K$ (see \cite{Gordon}).
In particular, the Dehn filling manifold $K(\pm3)$
is a \emph{small} Seifert fibered manifold, and
$K(\pm 3)$ fibers over $S^2(3,3,4)$.
The orbifold fundamental group 
$\pi_1^\mathcal{O}( S^2(3,3,4))$ is
isomorphic to the von Dyck group
$\pi_1^\mathcal{O}( S^2(3,3,4))\cong D(3,3,4)$.
Hence the  surjection
$
\Gamma\to\pi_1(K(\pm 3))\twoheadrightarrow\pi_1(K(\pm 3))/\mathrm{center}\cong D(3,3,4)
$ is natural.
\end{remark}

The center of $\pi_1(K(\pm 3))$ is generated by a regular fibre.
Any irreducible representation of $\pi_1(K(\pm 3))\to\SL(3)$
maps the fibre to the center of $\SL(3)$.
By using the description of $X_3(F_2)$ given by  Lawton \cite{Law07}) it quite elementary to determine  $X_3(D(3,3,4))$ explicitly. 
The proof of the next lemma can be found in \cite[Lemma 10.1]{HMP15}: 
\begin{lemma}\label{lem:D334}
The variety $\overline{X^{irr}(D(3,3,4),\mathrm{SL}(3,\CC) )}$ has a component $\mathcal{W}$ of dimension $2$ and three isolated points.
The variety $\mathcal{W}$ is isomorphic to the hypersurface in $\CC^3$ given by the equation
$$
\zeta^2-(\nu \bar{\nu}-2) \zeta+\nu^3+\bar{\nu}^3-5\nu\bar{\nu}+5=0\,.
$$
Here, the parameters are $\nu=\chi(k^{-1}l)$, $\bar{\nu}=\chi(k l^{-1})$ and $\zeta = \chi([k,l])$.
For every $\chi\in \mathcal{W}$, $\chi(k^{\pm 1})=\chi(l^{\pm 1})=0$ and $\chi( (kl)^{\pm 1})=1$.
 
Moreover, all characters in $\mathcal{W}$ are characters of  irreducible representations except for the three points
$(\nu,\bar{\nu},\zeta)=(2,2,1)$, $(2\varpi, 2\varpi^2, 1)$, $(2\varpi^2,2\varpi, 1)$, $\varpi=e^{2\pi i /3}$.
\end{lemma}

Now, the components $V_1$ and $V_2$ are  given by 
\[
V_1 = \phi^*(\mathcal{W}) \subset X_3(\Gamma_{4_1}) 
\quad\text{and }\quad
V_2 = (\phi\circ h)^*(\mathcal{W})\,.
\]
The components $V_1$ and $V_2$ are swapped by $h^*\co X_3(\Gamma_{4_1})\to X_3(\Gamma_{4_1})$, and $V_0$ is preserved.

Further details in the proof of  Lemma~\ref{lem:D334} allow to describe those  
three  isolated points.
Composing with $\phi^*$, they correspond to the three characters 
of irreducible metabelian representations in $X_3(\Gamma_{4_1})$ 
that do not lie in $V_2$. Altogether, there are five characters of irreducible metabelian representations (see \cite{Boden-Friedl2008}). All these metabelian characters are scheme smooth (see \cite{BF3}).
The character corresponding to a point of $V_0$ comes from
a surjection $\Gamma_{4_1}\twoheadrightarrow A_4$ composed with the irreducible
representation $A_4\to\SL(3)$.

\begin{proposition}\label{prop:non-dist-components}
 The components $V_1$ and $V_2$ are characters of representations which factor
 through the surjections $\Gamma\twoheadrightarrow \pi_1(K(\pm3))$ respectively.
These components are isomorphic to the hypersurface
 $$
 \zeta^2-(\nu \bar{\nu}-2) \zeta+ \nu^3+\bar{\nu}^3-5\nu\bar{\nu}+5=0.
 $$
 Here, the parameters are 
 \[ 
 \nu = 
 \begin{cases}
\chi(t) &\text {for $V_2$,}\\
\chi(t^{-1}) &\text {for $V_1$,}
 \end{cases}
\quad
 \bar{\nu} = 
 \begin{cases}
\chi(t^{-1}) &\text {for $V_2$,}\\
\chi(t) &\text {for $V_1$,}
 \end{cases}
 \quad
  \zeta = 
 \begin{cases}
\chi(a) &\text {for $V_2$,}\\
\chi(b^{-1}) &\text {for $V_1$.}
 \end{cases}
 \]
All characters are irreducible except for the three points
 $(\nu,\bar{\nu},\zeta)=(2,2,1),\ (2\varpi, 2\varpi^2, 1), (2\varpi^2,2\varpi, 1)$, with $\varpi=e^{2\pi i/3}$,
 that correspond to the intersection $V_1 \cap V_2=V_0\cap V_1 \cap V_2$.
\end{proposition}

{\small

}
\end{document}